\documentclass{mcom-l}

\usepackage{amssymb}
\usepackage{bm}

\usepackage{graphicx}

\newtheorem{theorem}{Theorem}[section]
\newtheorem{lemma}[theorem]{Lemma}

\theoremstyle{definition}
\newtheorem{definition}[theorem]{Definition}
\newtheorem{example}[theorem]{Example}

\theoremstyle{remark}

\numberwithin{equation}{section}

\newcommand{\llbracket}{\left[\!\left[}
\newcommand{\rrbracket}{\right] \! \right]}

\newcommand{\vertiii}[1]{\left| #1 
    \right|}

\begin{document}

\title{Generalized Korn's Inequalities for Piecewise $H^2$ Vector Fields}


\author{David M. Williams}
\address{Department of Mechanical Engineering, The Pennsylvania State University, Pennsylvania, United States, 16802}
\curraddr{}
\email{dmw72@psu.edu}
\thanks{}

\author{Qingguo Hong}
\address{Department of Mathematics, The Pennsylvania State University, Pennsylvania, United States, 16802}
\curraddr{}
\email{huq11@psu.edu}
\thanks{}


\date{}

\dedicatory{In memory of Philip Eric Gino Zwanenburg}

\begin{abstract}
The purpose of this paper is to construct a new class of discrete generalized Korn's inequalities for piecewise $H^2$ vector fields in three-dimensional space. The resulting Korn's inequalities are different from the standard Korn's inequalities, as they involve the trace-free symmetric gradient operator, in place of the usual symmetric gradient operator. It is anticipated that the new generalized Korn's inequalities will be useful for the analysis of a broad range of finite element methods, including  mixed finite element methods and discontinuous Galerkin methods.
\end{abstract}

\maketitle

\section{Introduction}

Broadly speaking, Korn's inequalities are an important part of the mathematical analysis of many partial differential equations and their associated finite element methods. In particular, the \emph{standard} Korn's inequalities provide an upper bound for the $L_2$ norm of the gradient of a sufficiently-smooth vector field in terms of the $L_2$ norm of the symmetric gradient and (sometimes) an auxiliary norm/seminorm. These types of Korn's inequalities can be classified into two general categories: i) continuous inequalities, and ii) discrete inequalities. The continuous inequalities are usually formulated for sufficiently smooth vector fields which are well-defined (with regard to smoothness) over the entirety of an open, bounded, Lipschitz domain. In contrast, discrete inequalities are usually formulated for sufficiently smooth vector fields which are well-defined in a piecewise fashion on Lipschitz subdomains of a larger domain. Naturally, the associated subdomains are usually convex, polyhedral elements which belong to a triangulation. The continuous class of inequalities can be used to analyze the non-discretized behavior of partial differential equations, or the behavior of numerical methods in which the entire computational domain is covered by a single element. One may consult~\cite{horgan1995korn} for a brief review of continuous Korn inequalities and their applications in the field of continuum mechanics. Conversely, the discrete class of inequalities can be used to analyze finite element methods which operate on the aforementioned partial differential equations. In particular, the discrete standard Korn's inequalities are useful for analyzing finite element methods for applications in linear elasticity~\cite{brenner2008mathematical,riviere2008discontinuous,di2013locking,di2015hybrid,hong2016robust}, and incompressible fluid dynamics~\cite{di2011mathematical,evans2013isogeometricA,evans2013isogeometricB,john2016finite,hong2016uniformly,fu2018strongly}. 

Recently, an interest in \emph{generalized} Korn's inequalities has emerged. The generalized Korn's inequalities are very similar to their standard counterparts, with the following distinction: they provide an upper bound for the gradient of a vector field in terms of the trace-free symmetric gradient instead of the standard symmetric gradient. The trace-free symmetric gradient naturally arises in the context of compressible fluid dynamics, where the trace of the velocity gradient tensor (i.e.~the divergence of the velocity field) is non-zero~\cite{anderson1990modern}. In addition, there are applications in the areas of general relativity, Cosserat elasticity, and geometry~\cite{fuchs2009application,schirra2009regularity,faraco2005geometric}. Several classes of continuous generalized Korn's inequalities have been previously developed by Dain~\cite{dain2006generalized}, Schirra~\cite{schirra2012new}, Feireisl and Novotn{\`y}~\cite{feireisl2009singular}, and Breit et al.~\cite{breit2017trace}. However, to the author's knowledge, there have been no rigorous efforts to develop discrete generalized Korn's inequalities. In principle, the development of such discrete inequalities will facilitate the construction of new, robust finite element methods in the aforementioned application areas. In addition, we note that there is an existing class of `versatile mixed finite element methods' for simulating incompressible/compressible fluid dynamics problems that will immediately benefit from the introduction of such inequalities~\cite{chen2020versatile,miller2020versatile}. In particular,  discrete generalized Korn's inequalities are required in order to construct rigorous error estimates for these methods. For all of the reasons mentioned above, this paper will attempt to derive the required class of generalized inequalities.  

In order to obtain the desired generalized inequalities, we will follow the approach of Brenner and coworkers~\cite{Brenner04,brenner2004poincare,brenner2015piecewise}. We note that other researchers have performed extensive work on constructing discrete standard Korn's inequalities, including Knobloch and Tobiska~\cite{knobloch2003korn}, Attia and Starke~\cite{attia2005discrete}, Mardal and Winther~\cite{mardal2006observation}, Zhang~\cite{zhang2017discrete}, and Lee~\cite{lee2018robust}; however, we choose to primarily focus on the work of Brenner in this paper. The interested reader is encouraged to consult~\cite{hong2016presentation} for a compact review of the other work on this topic. 

In~\cite{Brenner04}, Brenner constructs a class of discrete standard Korn's inequalities for piecewise $H^1$ vector fields in 2D and 3D. In~\cite{brenner2004poincare}, similar techniques are used to construct Poincar{\'e}--Friedrichs discrete inequalities for piecewise $H^2$ vector fields in 2D. Finally, in~\cite{brenner2015piecewise}, standard Korn and Poincar{\'e}--Friedrichs discrete inequalities are obtained for classes of meshes generated by independent refinements, for piecewise $H^1$ vector fields in 2D and 3D. 

We intend for the present work to be a natural extension of~\cite{Brenner04} to the case of generalized Korn's inequalities, and piecewise $H^2$ vector fields in 3D. The organization of this paper is as follows. In section 2, we provide an overview of generalized Korn's inequalities at the continuous level. In section 3, we present several new generalized Korn's inequalities at the discrete level. In section 4, we provide proofs for the discrete inequalities. Finally, in section 5, we conclude with a brief summary of the present work. 

\section{The Continuous Case}

Let us begin by introducing some important notation and concepts which will be beneficial during the remainder of our presentation.

One may consider a bounded, connected, open, polyhedral domain $\Omega$ in $\mathbb{R}^3$, with boundary $\partial \Omega$. By construction, the boundary of this domain is Lipschitz. On this domain, we define a generic vector field $\bm{u} \in \left[H^2 (\Omega)\right]^3$. In a natural fashion, we define the following norms and seminorms which produce finite results when acting upon $\bm{u}$
\begin{align*}
    \left\| \bm{u} \right\|_{L_2 \left(\Omega\right)} &= \left( \int_{\Omega} | \bm{u} |^2 \, dx \right)^{1/2}, \\[1.5ex]
    \left|\bm{u}\right|_{H^{1}(\Omega)}^{2} &= \sum_{j=1}^{3} \left| u_j \right|^{2}_{H^{1} (\Omega)}, \qquad
    \left\| \bm{u} \right\|_{H^1(\Omega)}^{2} = \left\| \bm{u} \right\|_{L_2 \left(\Omega\right)}^2 + \left| \bm{u} \right|_{H^{1}(\Omega)}^{2},
    \\[1.5ex]
    \left|\bm{u}\right|_{H^{2}(\Omega)}^{2} &= \sum_{j=1}^{3} \left| u_j \right|^{2}_{H^{2} (\Omega)}, \qquad
    \left\| \bm{u} \right\|_{H^2(\Omega)}^{2} = \left\| \bm{u} \right\|_{L_2 \left(\Omega\right)}^2 + \left| \bm{u} \right|_{H^{1}(\Omega)}^{2} + \left| \bm{u} \right|_{H^{2}(\Omega)}^{2},
\end{align*}
where $|\bm{u}|$ is the standard Euclidean norm of $\bm{u}$, and 
\begin{align*}
    \left| u_j \right|_{H^{k}(\Omega)} = \left( \sum_{|\bm{\alpha}| = k} \int_{\Omega} \left| D^{\bm{\alpha}} u_j \right|^2 \, dx \right)^{1/2},
\end{align*}
is the Hilbert seminorm of $u_j$ for $H^k (\Omega)$. In the equation above, $\bm{\alpha} = \left(\alpha_1, \alpha_2, \alpha_3\right)$ is the multi-index, and $D^{\bm{\alpha}}$ is the generalized, multi-dimensional derivative operator. 

Next, we introduce the following, more specialized, derivative operators
\begin{align*}
    \epsilon_{ij} \left(\bm{u}\right) &= \frac{1}{2} \left(\text{grad}_{ij}(\bm{u}) + \text{grad}_{ji}(\bm{u}) \right) = \frac{1}{2} \left( \frac{\partial u_i}{\partial x_j} + \frac{\partial u_j}{\partial x_i} \right), \\[1.5ex]
    \text{div} \left(\bm{u}\right) &= \sum_{i=1}^{3} \frac{\partial u_i}{\partial x_i}, \\[1.5ex]
    \text{curl} \left(\bm{u} \right) &= \left(\frac{\partial u_3}{\partial x_2} - \frac{\partial u_2}{\partial x_3}, \frac{\partial u_1}{\partial x_3} - \frac{\partial u_3}{\partial x_1}, \frac{\partial u_2}{\partial x_1} - \frac{\partial u_1}{\partial x_2} \right)^t,
\end{align*}
where $\bm{\epsilon} \left( \cdot \right)$ is the symmetric gradient operator, $\text{grad}(\cdot)$ is the standard gradient operator, $\text{div}\left(\cdot \right)$ is the divergence operator, and $\text{curl} \left(\cdot \right)$ is the curl operator. In addition, we introduce the trace-free symmetric gradient operator, which takes the following form
\begin{align*}
    \bm{\epsilon} \left(\bm{u}\right) - \frac{1}{3} \text{div} \left(\bm{u} \right) \mathbb{I}, 
\end{align*}
where $\mathbb{I}$ is the $3\times3$ identity matrix. Now, we can define the following norms which remain finite for all $\bm{u} \in \left[H^{1}(\Omega) \right]^3$
\begin{align}
    &\left\| \bm{\epsilon} \left(\bm{u}\right) \right\|_{L_2(\Omega) \times L_2(\Omega)} = \left( \int_{\Omega} \left| \bm{\epsilon} \left(\bm{u}\right) \right|^{2} \, dx \right)^{1/2}, \label{norm_one} \\[1.5ex]
    &\left\| \bm{\epsilon} \left(\bm{u}\right) - \frac{1}{3} \text{div} \left(\bm{u} \right) \mathbb{I} \right\|_{L_2(\Omega) \times L_2(\Omega)} = \left( \int_{\Omega} \left| \bm{\epsilon} \left(\bm{u}\right) - \frac{1}{3} \text{div} \left(\bm{u} \right) \mathbb{I} \right|^{2} \, dx \right)^{1/2}, \label{norm_two}
\end{align}
where the quantity $\left| \bm{\epsilon} \left(\bm{u}\right) \right|$ is the standard Frobenius inner-product applied to $\bm{\epsilon} (\bm{u})$. 

We note that if $\bm{u}$ is pointwise divergence-free, then the norm definitions in Eqs.~\eqref{norm_one} and \eqref{norm_two} are identical, and the necessary analysis has already been previously established in other work, (see the Introduction). However, for the case of a generic $\bm{u}$ which is not divergence-free, there are some important results which remain to be investigated. 

Towards this end, one may introduce the kernel space of the trace-free symmetric gradient
\begin{align*}
    \mathbf{CK} \left(\Omega\right) = \left\{2 \left(\bm{a} \cdot \bm{x} \right) \bm{x} - \left| \bm{x} \right|^2 \bm{a} + \rho \bm{x} + \bm{Q} \bm{x} + \bm{b} : \bm{a} \in \mathbb{R}^{3}, \bm{b} \in \mathbb{R}^{3}, \rho \in \mathbb{R}, \; \text{and} \; \bm{Q} \in \mathfrak{so}\left(3\right)   \right\},
\end{align*}
where $\bm{x} = \left(x_1, x_2, x_3\right)$ is a coordinate vector, and $\mathfrak{so}(3)$ is the space of $3\times 3$ skew-symmetric matrices. The kernel space (above) is identical to the space of conformal Killing vectors, which was originally identified in~\cite{dain2006generalized} and further studied in~\cite{schirra2012new}. This space contains infinitesimal rigid body motions, Killing vectors, and dilatations. By construction, we have that
\begin{align}
    \bm{\epsilon} \left(\bm{u}\right) - \frac{1}{3} \text{div} \left(\bm{u}\right) \mathbb{I} = \bm{0} \qquad \Longleftrightarrow \qquad \bm{u} \in \mathbf{CK} \left(\Omega\right). \label{kernel_condition}
\end{align}
It turns out that proving Korn-type inequalities for the trace-free symmetric gradient is less straightforward than proving the equivalent inequalities for the standard symmetric gradient. The main challenges arise from the fact that the kernel space for the symmetric gradient is simpler than $\mathbf{CK} (\Omega)$, as it is merely the space of infinitesimal rigid body motions
\begin{align*}
    \mathbf{RM}(\Omega) = \left\{\bm{Q} \bm{x} + \bm{b} : \bm{b} \in \mathbb{R}^{3} \; \text{and} \; \bm{Q} \in \mathfrak{so}\left(3\right)  \right\}.
\end{align*}
The key complication is that $\mathbf{RM}(\Omega)$ is linear in $\bm{x}$, whereas the space $\mathbf{CK} (\Omega)$ is quadratic. 

With the above discussion in mind, it is somewhat surprising to find that Korn's inequalities for the trace-free symmetric gradient exist. In particular, in accordance with~\cite{feireisl2009singular} (Theorem 11.22, pg. 466), we have that 
\begin{align*}
    \left\| \bm{u} \right\|_{H^{1} (\Omega)} \leq C \left( \left\| \bm{\epsilon} \left(\bm{u}\right) - \frac{1}{3} \text{div} \left(\bm{u} \right) \mathbb{I} \right\|_{L_2(\Omega) \times L_2(\Omega)} + \int_{\Omega} \left| \bm{u} \right| \, dx \right), \qquad \forall \bm{u} \in \left[ H^{1} (\Omega) \right]^{3}.
\end{align*}
Now, after some simple manipulations with H{\"o}lder's inequality, we obtain
\begin{align}
    \left\| \bm{u} \right\|_{H^{1} (\Omega)} \leq C_{\Omega} \left( \left\| \bm{\epsilon} \left(\bm{u}\right) - \frac{1}{3} \text{div} \left(\bm{u} \right) \mathbb{I} \right\|_{L_2(\Omega) \times L_2(\Omega)} + \left\| \bm{u} \right\|_{L_2(\Omega)} \right), \label{korn_basic_cont_full}
\end{align}
and
\begin{align}
    \left| \bm{u} \right|_{H^{1} (\Omega)} \leq C_{\Omega} \left( \left\| \bm{\epsilon} \left(\bm{u}\right) - \frac{1}{3} \text{div} \left(\bm{u} \right) \mathbb{I} \right\|_{L_2(\Omega) \times L_2(\Omega)} + \left\| \bm{u} \right\|_{L_2(\Omega)} \right). \label{korn_basic_cont}
\end{align}
Here, the Korn's inequality has been rewritten in the standard fashion: namely, it provides a bound for the gradient norm of $\bm{u}$ in terms of the trace-free gradient norm and the $L_2$ norm. It remains for us to specialize this inequality, and replace the $L_2$ norm with a broader class of more useful seminorms. Towards this end, we define a class of seminorms on $\left[ H^1(\Omega) \right]^{3}$ that satisfy the following constraints
\begin{align}
    \bm{\Phi} \left(\bm{v}\right) \leq C_{\bm{\Phi}} \left\| \bm{v} \right\|_{H^1(\Omega)}, \qquad \forall \bm{v} \in \left[ H^1(\Omega) \right]^{3}, \label{phi_cond_one}
\end{align}
and
\begin{align}
    \bm{\Phi} \left(\bm{m}\right) = 0 \qquad \text{and} \qquad \bm{m} \in \mathbf{CK}(\Omega) \qquad \Longleftrightarrow \qquad \bm{m} = \text{a constant vector}, \label{phi_cond_two}
\end{align}
where $C_{\bm{\Phi}}$ is a constant that depends on our particular choice of seminorm. It turns out that all seminorms that satisfy the second criterion, will also be invariant under the addition of a constant vector, such that 
\begin{align}
    \bm{\Phi} \left(\bm{v} + \bm{c} \right) = \bm{\Phi} \left(\bm{v} \right), \label{phi_cond_three}
\end{align}
where $\bm{c} \in \mathbb{R}^{3}$. 

Now, based on the seminorm properties given above, and Eq.~\eqref{korn_basic_cont}, it can be shown that
\begin{align}
    \left| \bm{v} \right|_{H^{1} (\Omega)} \leq C_{\Omega} \left( \left\| \bm{\epsilon} \left(\bm{v}\right) - \frac{1}{3} \text{div} \left(\bm{v} \right) \mathbb{I} \right\|_{L_2(\Omega) \times L_2(\Omega)} + \bm{\Phi} \left(\bm{v}\right) \right), \qquad \forall \bm{v} \in \left[ H^1(\Omega) \right]^{3}.
    \label{korn_seminorm_cont}
\end{align}
The proof of Eq.~\eqref{korn_seminorm_cont} uses the fact that $H^{1}$ is compactly embedded into $L_2$. Some details of the proof are given in Appendix A, Lemma~\ref{seminorm_lemma}. 

There are an infinite number of seminorms which satisfy the required constraints (Eqs.~\eqref{phi_cond_one} and \eqref{phi_cond_two}). However, we choose to focus on the following seminorms due to their utility for practical applications
\begin{align}
    \label{seminorm_def_one} \bm{\Phi}_{1} \left(\bm{v} \right) &= \left\| G \bm{v} \right\|_{L_2(\Omega)}, \\[1.5ex]
    \bm{\Phi}_{2} \left(\bm{v} \right) &= \sup_{ \substack{\bm{m} \in \mathbf{CK}(\Omega) \\ \label{seminorm_def_two}
    \left\| \bm{m} \right\|_{L_2(\partial \Omega)} = 1, \int_{\partial \Omega} \bm{m} \, ds = \bm{0}}} \int_{\partial \Omega} \bm{v} \cdot \bm{m} \, ds,
\end{align}
where
\begin{align*}
    G \bm{v} = \bm{v} - \frac{1}{|\Omega|} \int_{\Omega} \bm{v} \, dx,
\end{align*}
is the projection from $\left[L_{2}(\Omega)\right]^{3}$ on to the orthogonal complement of constant vector fields. Generally speaking, we have that
\begin{align}
    \left\| G\bm{v} \right\|_{L_2(\Omega)} \leq C \left\| \bm{v} \right\|_{L_2(\Omega)}, \qquad \forall \bm{v} \in \left[L_2(\Omega)\right]^{3}. \label{ortho_bound}
\end{align}

\section{The Discrete Case}

Our domain $\Omega \subset \mathbb{R}^3$ can tessellated by non-overlapping tetrahedral elements. We denote the individual elements by $T$, and the overall partition of the domain into these elements by $\mathcal{T}$. The resulting partition has the following properties:
\begin{align*}
    \overline{\Omega} = \bigcup_{T \in \mathcal{T}} \overline{T},
\end{align*}
and 
\begin{align*}
    T \cap T' = \emptyset \qquad \text{unless} \qquad T = T'. 
\end{align*}
We assume that the minimum interior angle of any tetrahedron $T \in \mathcal{T}$ is given by $\theta_{\mathcal{T}}$. In addition, we assume that the partition is conforming, such that all interior faces are shared by two tetrahedra, and all boundary faces are associated with only one tetrahedron. We denote the individual faces in the partition by $\sigma$. The total set of interior faces is denoted by $S\left(\mathcal{T}, \Omega\right)$, and for each $\sigma \in S\left(\mathcal{T}, \Omega\right)$, we can define a ``$-$" side and a ``$+$" side. Next, we define a jump operator
\begin{align*}
    \llbracket \bm{u} \rrbracket_{\sigma} = \bm{u}_{+} - \bm{u}_{-},
\end{align*}
where $\bm{u}$ is a generic vector that has well-defined values on the interior faces. It is also convenient to define a quadratic projection operator $\pi_{\sigma}$ on each face, such that $\pi_{\sigma}: \left[L_2(\sigma)\right]^3 \rightarrow \left[P_2(\sigma)\right]^3$. This operator satisfies the following standard projection inequality
\begin{align}
    \left\| \pi_{\sigma} \llbracket \bm{u} \rrbracket_{\sigma} \right\|_{L_2(\sigma)} \leq C \left\| \llbracket \bm{u} \rrbracket_{\sigma} \right\|_{L_2(\sigma)}, \qquad \forall \bm{u} \in \left[L_2(\sigma)\right]^{3}. \label{proj_ineq}
\end{align}
The definitions (above) allow us to effectively characterize vector fields which possess discontinuities along the interior faces. We are particularly interested in the class of (possibly discontinuous) piecewise $H^2$ vector fields $\bm{u} \in \left[H^2(\Omega, \mathcal{T})\right]^{3}$, where
\begin{align*}
    \left[H^2(\Omega, \mathcal{T})\right]^{3} = \left\{ \bm{u} \in \left[ L_{2}(\Omega) \right]^{3} : \bm{v}_{T} = \bm{v}|_{T} \in \left[H^{2}(T) \right]^{3} \quad  \forall T \in \mathcal{T} \right\}.
\end{align*}
The associated broken seminorms can be defined as follows
\begin{align*}
    \left| \bm{u} \right|_{H^{1}(\Omega, \mathcal{T})} = \left( \sum_{T \in \mathcal{T}} \left| \bm{u} \right|^{2}_{H^1(T)} \right)^{1/2}, \qquad  \left| \bm{u} \right|_{H^{2}(\Omega, \mathcal{T})} = \left( \sum_{T \in \mathcal{T}} \left| \bm{u} \right|^{2}_{H^2(T)} \right)^{1/2}.
\end{align*}
In addition, we can define piecewise derivative operators $\forall T \in \mathcal{T}$
\begin{align*}
    \bm{\epsilon}_{\mathcal{T}} (\bm{u}) |_{T} = \bm{\epsilon} \left( \bm{u}_{T} \right), \qquad \text{div}_{\mathcal{T}} (\bm{u}) |_{T} = \text{div} \left(\bm{u}_{T} \right). 
\end{align*}
%

With the above discussion in mind, we can formally state a new class of generalized, discrete Korn's inequalities for all $\bm{v} \in \left[H^2(\Omega, \mathcal{T})\right]^{3}$
\begin{align}
    \left| \bm{v} \right|_{H^{1} (\Omega, \mathcal{T})}^{2} \leq C \Bigg(& \left\| \bm{\epsilon}_{\mathcal{T}} (\bm{v}) - \frac{1}{3} \text{div}_{\mathcal{T}} (\bm{v}) \mathbb{I} \right\|_{L_2(\Omega) \times L_2(\Omega)}^{2} + \left(\bm{\Phi} (\bm{v}) \right)^{2} \label{korn_seminorm_disc_proj} \\[1.5ex]
    \nonumber &+ \sum_{\sigma \in S(\mathcal{T}, \Omega)} \left(\text{diam} \, \sigma\right)^{-1} \left\| \pi_{\sigma} \llbracket \bm{v} \rrbracket_{\sigma} \right\|^{2}_{L_2(\sigma)} \Bigg).
\end{align}
In accordance with Eq.~\eqref{proj_ineq}, we can immediately eliminate the dependence of the inequality on the projection operator $\pi_{\sigma}$, and obtain
\begin{align}
        \left| \bm{v} \right|_{H^{1} (\Omega, \mathcal{T})}^{2} \leq C \Bigg(& \left\| \bm{\epsilon}_{\mathcal{T}} (\bm{v}) - \frac{1}{3} \text{div}_{\mathcal{T}} (\bm{v}) \mathbb{I} \right\|_{L_2(\Omega) \times L_2(\Omega)}^{2} + \left(\bm{\Phi} (\bm{v}) \right)^{2} \label{korn_seminorm_disc} \\[1.5ex]
    \nonumber &+ \sum_{\sigma \in S(\mathcal{T}, \Omega)} \left(\text{diam} \, \sigma\right)^{-1} \left\| \llbracket \bm{v} \rrbracket_{\sigma} \right\|^{2}_{L_2(\sigma)} \Bigg).
\end{align}
The Korn's inequalities above can be expressed explicitly for each seminorm, as follows
\begin{align}
        \left| \bm{v} \right|_{H^{1} (\Omega, \mathcal{T})}^{2} \leq C \Bigg(& \left\| \bm{\epsilon}_{\mathcal{T}} (\bm{v}) - \frac{1}{3} \text{div}_{\mathcal{T}} (\bm{v}) \mathbb{I} \right\|_{L_2(\Omega) \times L_2(\Omega)}^{2} + \left\| G \bm{v} \right\|_{L_2(\Omega)}^{2} \\[1.5ex]
    \nonumber &+ \sum_{\sigma \in S(\mathcal{T}, \Omega)} \left(\text{diam} \, \sigma\right)^{-1} \left\| \llbracket \bm{v} \rrbracket_{\sigma} \right\|^{2}_{L_2(\sigma)} \Bigg), \\[1.5ex]
           \left| \bm{v} \right|_{H^{1} (\Omega, \mathcal{T})}^{2} \leq C \Bigg(& \left\| \bm{\epsilon}_{\mathcal{T}} (\bm{v}) - \frac{1}{3} \text{div}_{\mathcal{T}} (\bm{v}) \mathbb{I} \right\|_{L_2(\Omega) \times L_2(\Omega)}^{2} + \sup_{ \substack{\bm{m} \in \mathbf{CK}(\Omega) \\ 
    \left\| \bm{m} \right\|_{L_2(\partial \Omega)} = 1,  \int_{\partial \Omega} \bm{m} \, ds = \bm{0}}} \left( \int_{\partial \Omega} \bm{v} \cdot \bm{m} \, ds \right)^{2}, \\[1.5ex]
     \nonumber &+ \sum_{\sigma \in S(\mathcal{T}, \Omega)} \left(\text{diam} \, \sigma\right)^{-1} \left\| \llbracket \bm{v} \rrbracket_{\sigma} \right\|^{2}_{L_2(\sigma)} \Bigg), 
\end{align}
or equivalently,
\begin{align}
        \left| \bm{v} \right|_{H^{1} (\Omega, \mathcal{T})}^{2} \leq C \Bigg(& \left\| \bm{\epsilon}_{\mathcal{T}} (\bm{v}) - \frac{1}{3} \text{div}_{\mathcal{T}} (\bm{v}) \mathbb{I} \right\|_{L_2(\Omega) \times L_2(\Omega)}^{2} + \left\| G \bm{v} \right\|_{L_2(\Omega)}^{2} \label{useful_one} \\[1.5ex]
    \nonumber &+ \sum_{\sigma \in S(\mathcal{T}, \Omega)} \left(\text{diam} \, \sigma\right)^{-1} \left\| \llbracket \bm{v} \rrbracket_{\sigma} \right\|^{2}_{L_2(\sigma)} \Bigg), \\[1.5ex]
           \left| \bm{v} \right|_{H^{1} (\Omega, \mathcal{T})}^{2} \leq C \Bigg(& \left\| \bm{\epsilon}_{\mathcal{T}} (\bm{v}) - \frac{1}{3} \text{div}_{\mathcal{T}} (\bm{v}) \mathbb{I} \right\|_{L_2(\Omega) \times L_2(\Omega)}^{2} + \left\| \bm{v} \right\|^{2}_{L_2(\partial \Omega)} \label{useful_two} \\[1.5ex]
    \nonumber &+ \sum_{\sigma \in S(\mathcal{T}, \Omega)} \left(\text{diam} \, \sigma\right)^{-1} \left\| \llbracket \bm{v} \rrbracket_{\sigma} \right\|^{2}_{L_2(\sigma)} \Bigg).  
\end{align}
The inequalities in Eqs.~\eqref{useful_one} and \eqref{useful_two} are the most useful in practice. In particular, they can be used for mixed finite element methods and discontinuous Galerkin methods. They can also be used for continuous Galerkin methods. In this latter case, all of the jump terms vanish, and one obtains the following simplified inequalities  
\begin{align}
        \left| \bm{v} \right|_{H^{1} (\Omega, \mathcal{T})}^{2} \leq C \Bigg(& \left\| \bm{\epsilon}_{\mathcal{T}} (\bm{v}) - \frac{1}{3} \text{div}_{\mathcal{T}} (\bm{v}) \mathbb{I} \right\|_{L_2(\Omega) \times L_2(\Omega)}^{2} + \left\| G \bm{v} \right\|_{L_2(\Omega)}^{2} \Bigg), \\[1.5ex]
           \left| \bm{v} \right|_{H^{1} (\Omega, \mathcal{T})}^{2} \leq C \Bigg(& \left\| \bm{\epsilon}_{\mathcal{T}} (\bm{v}) - \frac{1}{3} \text{div}_{\mathcal{T}} (\bm{v}) \mathbb{I} \right\|_{L_2(\Omega) \times L_2(\Omega)}^{2} + \left\| \bm{v} \right\|^{2}_{L_2(\partial \Omega)} \Bigg). 
\end{align}
It remains for us to prove the key discrete result (Eq.~\eqref{korn_seminorm_disc_proj}), from whence all other discrete results follow. This will be the focus of the next section. 

\section{Proofs of Discrete Results}
In this section, we will construct a rigorous proof for Eq.~\eqref{korn_seminorm_disc_proj}. This will require the introduction of two lemmas, prior to the final result. The final result itself will be stated in a theorem at the end of the section. 

Prior to the first lemma, it is necessary to introduce a mapping operator, along with some additional notation. Let us denote the set of all vertices in $\mathcal{T}$ as $\mathcal{V} (\mathcal{T})$, and the set of all edge midpoints as $\mathcal{M} (\mathcal{T})$. Thereafter, one may define a mapping operator $E$ which maps piecewise quadratic vector fields to continuous piecewise quadratic vector fields. More precisely, we define $E: V_{\mathcal{T}} \rightarrow W_{\mathcal{T}}$, where
\begin{align*}
    V_{\mathcal{T}} &= \left\{\bm{v} \in \left[L_2(\Omega)\right]^{3}: \bm{v}_{T} = \bm{v}|_{T} \in \left[P_2(T) \right]^{3} \quad \forall T \in \mathcal{T} \right\}, \\[1.5ex]
    W_{\mathcal{T}} &= \left\{\bm{w} \in \left[H^1(\Omega)\right]^{3}: \bm{w}_{T} = \bm{w}|_{T} \in \left[P_2(T) \right]^{3} \quad \forall T \in \mathcal{T} \right\},
\end{align*}
and where
\begin{align}
    \left(E \bm{v}\right)(p) &= \frac{1}{|\mathcal{T}_p|} \sum_{T \in \mathcal{T}_{p}} \bm{v}_{T} (p), \qquad \forall p \in \mathcal{V}(\mathcal{T}), \label{edef_one} \\[1.5ex]
    \left(E \bm{v} \right)(m) &= \frac{1}{|\mathcal{T}_m|} \sum_{T \in \mathcal{T}_{m}} \bm{v}_{T} (m), \qquad \forall m \in \mathcal{M}(\mathcal{T}). \label{edef_two}
\end{align}
Here, we define $\mathcal{T}_p$ as the set of tetrahedra sharing a common vertex $p$
\begin{align*}
    \mathcal{T}_p = \left\{ T \in \mathcal{T} : p \in \partial T \right\},
\end{align*}
and $\mathcal{T}_m$ as the set of tetrahedra sharing a common edge with midpoint $m$
\begin{align*}
    \mathcal{T}_m = \left\{ T \in \mathcal{T} : m \in \partial T \right\}.
\end{align*}
Figure~\ref{tet_chain_fig} provides an example of a shared vertex and a shared edge midpoint.
\begin{figure}[h!]
\includegraphics[width=12cm]{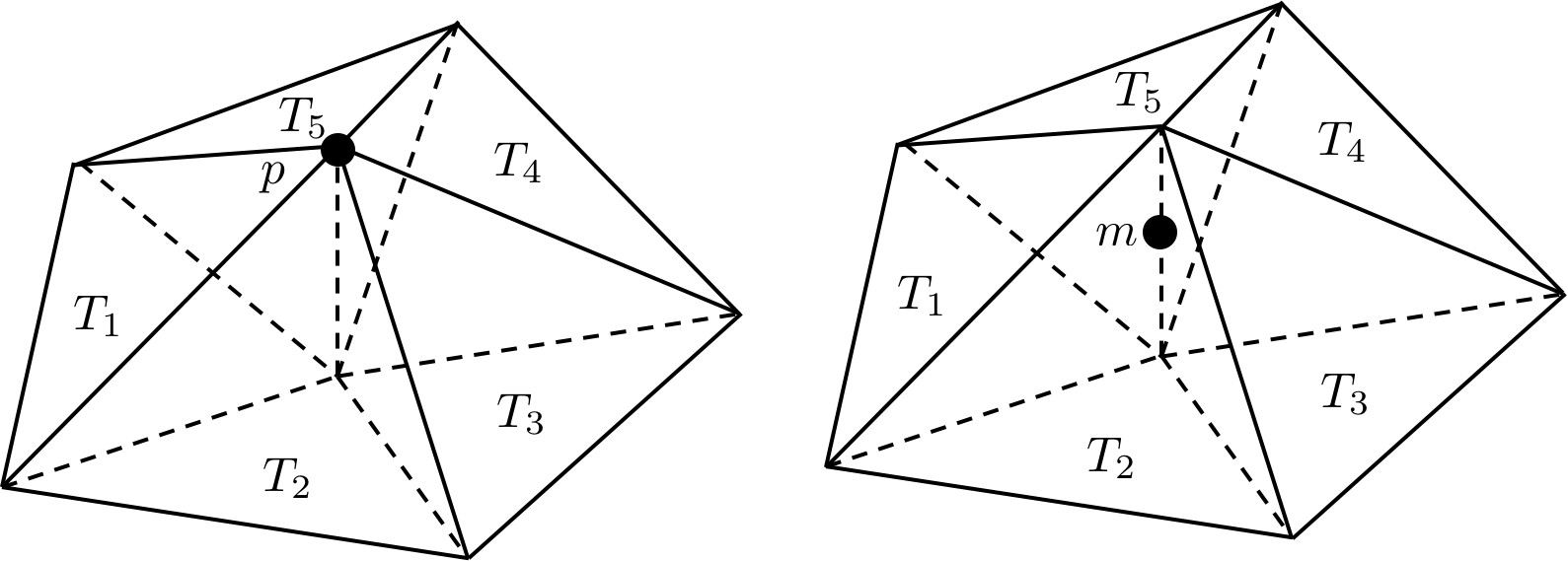}
\caption{A collection of tetrahedra: $T_1, T_2, \ldots, T_5$. The tetrahedra share two common vertices, one of which ($p$) is shown on the left. The tetrahedra also share a single edge midpoint ($m$) which is shown on the right.}
\label{tet_chain_fig}
\end{figure}
In general, we assume that the number of tetrahedra sharing a vertex or midpoint is relatively small, i.e.
\begin{align*}
    | \mathcal{T}_{p} | &\lesssim 1, \qquad \forall p \in \mathcal{V}(\mathcal{T}), \\[1.5ex]
    | \mathcal{T}_{m} | &\lesssim 1, \qquad \forall m \in \mathcal{M}(\mathcal{T}). 
\end{align*}
Here, and throughout the remainder of this section, we say that $A \lesssim B$ when $A \leq c_0 B$, and $c_0$ is a constant do not depend on the mesh. In a natural fashion, we say that $A \lesssim B$ and $B \lesssim A$ implies $A \approx B$. 

In addition, 
the continuous piecewise quadratic vector fields (which reside in $W_{\mathcal{T}})$ can be represented using piecewise quadratic, continuous Galerkin finite elements. These elements can be visualized as shown in figure~\ref{p2_fig}.
\begin{figure}[h!]
\includegraphics[width=5cm]{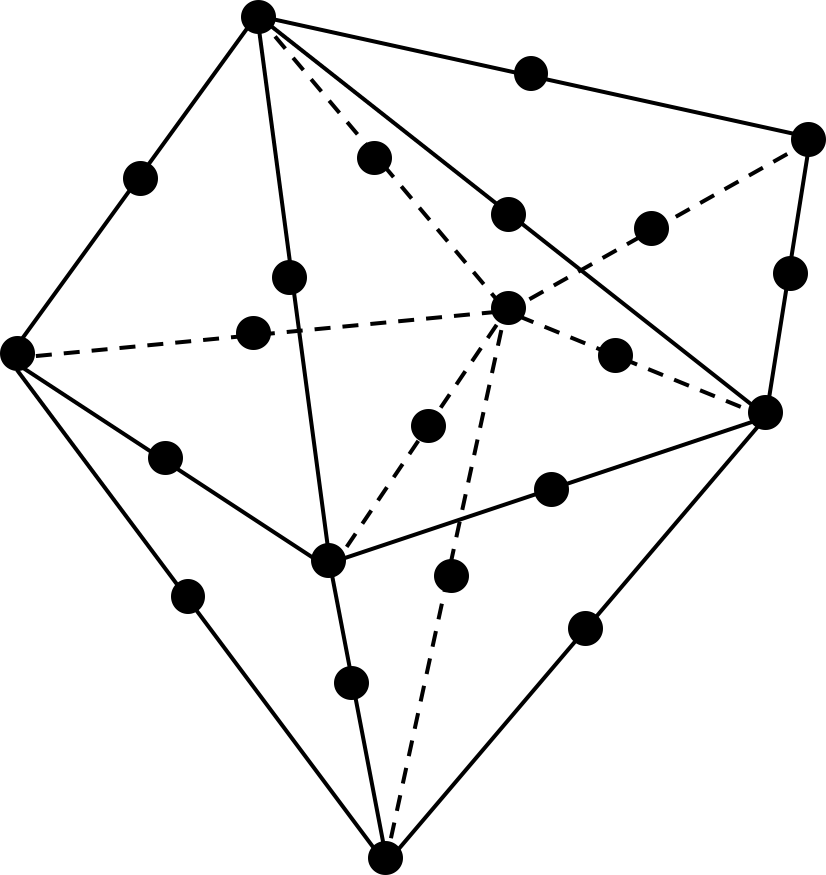}
\caption{A collection of continuous, piecewise quadratic, finite elements. The degrees of freedom for each element are represented using solid dots.}
\label{p2_fig}
\end{figure}

We are now ready to introduce the first lemma of this section.

\begin{lemma}
The pointwise differences between the piecewise vector fields $\bm{v} \in V_{\mathcal{T}}$ and the continuous mapping $E \bm{v} \in W_{\mathcal{T}}$ are given by the following inequalities
\begin{align}
    \left| \bm{v}_{T} (p) - E \bm{v} (p) \right|^2 &\lesssim \sum_{\sigma \in \Xi_p} \left| \llbracket \bm{v} \rrbracket_{\sigma} (p) \right|^2, \qquad \forall T \in \mathcal{T}, \; \forall p \in \mathcal{V} (\mathcal{T}), \label{jump_one} \\[1.5ex]
    \left| \bm{v}_{T} (m) - E \bm{v} (m) \right|^2 &\lesssim \sum_{\sigma \in \Xi_m} \left| \llbracket \bm{v} \rrbracket_{\sigma} (m) \right|^2, \qquad \forall T \in \mathcal{T}, \; \forall m \in \mathcal{M} (\mathcal{T}), \label{jump_two}
\end{align}
where $\Xi_p$ is the set of interior faces sharing $p$ as a common vertex
\begin{align*}
    \Xi_{p} = \left\{ \sigma \in S(\mathcal{T},\Omega) : p \in \partial \sigma \right\},
\end{align*}
and $\Xi_m$ is the set of interior faces sharing $m$ as a common edge midpoint
\begin{align*}
    \Xi_{m} = \left\{ \sigma \in S(\mathcal{T},\Omega) : m \in \partial \sigma \right\}.
\end{align*}
\end{lemma}

\begin{proof}
The proof of Eq.~\eqref{jump_one} appears in Lemma 2.1 of~\cite{Brenner04}. It remains for us to prove Eq.~\eqref{jump_two}. We begin by using the definition of $E$ from Eq.~\eqref{edef_two} in order to obtain the following expression
\begin{align}
    \bm{v}_{T} (m) - E\bm{v} (m) = \frac{1}{|\mathcal{T}_m|} \sum_{T' \in \mathcal{T}_m} \left(\bm{v}_{T}(m) - \bm{v}_{T'}(m) \right).
    \label{lemma_1a}
\end{align}
Next, we introduce a sequence of tetrahedra which share the edge midpoint $m$, where $T = T_1$ and $T' = T_k$, and $k > 1$. An example of such a sequence is $T_1, T_2, \ldots, T_{j}, T_{j+1}, \ldots, T_{k-1}, T_k$, where $k \leq |\mathcal{T}_{m}|$. Furthermore, suppose that $T_j$ and $T_{j+1}$ share a common face $\sigma_j \in \Xi_m$. Under these circumstances, the following identity holds 
\begin{align}
    \left| \bm{v}_{T}(m) - \bm{v}_{T'}(m) \right|^{2} &= \left| \sum_{j=1}^{k-1}\left(\bm{v}_{T_j}(m) - \bm{v}_{T_{j+1}}(m) \right) \right|^2 \label{lemma_1b} \\[1.5ex]
    \nonumber & \lesssim \sum_{j=1}^{k-1} \left| \llbracket \bm{v} \rrbracket_{\sigma_j}(m) \right|^2,
\end{align}
where we have used the root-mean-square-arithmetic-mean inequality on the last line. Upon combining Eqs.~\eqref{lemma_1a} and \eqref{lemma_1b}, and using the root-mean-square-arithmetic-mean inequality a second time, we obtain the desired result (Eq.~\eqref{jump_two}).
\end{proof}

\begin{lemma}
Suppose that a seminorm $\bm{\Phi}$ satisfies the following inequalities
\begin{align}
    \bm{\Phi} (\bm{w}) &\leq C_{\bm{\Phi}} \left\| \bm{w} \right\|_{H^1(\Omega, \mathcal{T})}, \qquad \bm{w} \in \left[ H^{1}(\Omega, \mathcal{T}) \right]^{3}, \\[1.5ex]
    \left(\bm{\Phi}(\bm{v} - E \bm{v} )\right)^{2} & \lesssim \sum_{\sigma \in S(\mathcal{T},\Omega)} (\text{diam} \, \sigma) \left( \sum_{p \in \mathcal{V}(\sigma)} \left| \llbracket \bm{v} \rrbracket_{\sigma} (p) \right|^{2} + \sum_{m \in \mathcal{M}(\sigma)} \left| \llbracket \bm{v} \rrbracket_{\sigma} (m) \right|^{2} \right), \quad \forall \bm{v} \in V_{\mathcal{T}}, \label{seminorm_bound}
\end{align}
in conjunction with
\begin{align*}
    \bm{\Phi} (\bm{m}) = 0 \qquad \text{and} \qquad \bm{m} \in \mathbf{CK}(\Omega) \qquad \Longleftrightarrow \qquad \bm{m} = \text{const},
\end{align*}
where $\mathcal{V}(\sigma)$ is the set of all vertices of $\sigma$, and $\mathcal{M}(\sigma)$ is the set of all edge midpoints of $\sigma$. Under these circumstances, the following inequality holds for all $\bm{v} \in V_{\mathcal{T}}$
\begin{align}
    \left| \bm{v} \right|^{2}_{H^{1}(\Omega, \mathcal{T})} & \lesssim \left\| \epsilon_{\mathcal{T}}(\bm{v}) - \frac{1}{3} \text{div}_{\mathcal{T}} (\bm{v}) \mathbb{I} \right\|^{2}_{L_2(\Omega) \times L_2(\Omega)} + \left(\bm{\Phi}(\bm{v})\right)^{2} \label{v_estimate} \\[1.5ex]
    \nonumber &+\sum_{\sigma \in S(\mathcal{T},\Omega)} (\text{diam} \, \sigma) \left( \sum_{p \in \mathcal{V}(\sigma)} \left| \llbracket \bm{v} \rrbracket_{\sigma} (p) \right|^{2} + \sum_{m \in \mathcal{M}(\sigma)} \left| \llbracket \bm{v} \rrbracket_{\sigma} (m) \right|^{2} \right).
\end{align}
\end{lemma}

\begin{proof}
    We begin by assuming that the diameter of each triangle is comparable in size to the diameters of the nearby faces 
    \begin{align}
        \text{diam} \, T \approx \text{diam} \, \sigma, \qquad &\forall T \in \mathcal{T}_{p}, \; \sigma \in \Xi_{p}, \; p \in \mathcal{V}(\mathcal{T}), \label{diam_criterion} \\[1.5ex]
        \nonumber & \forall  T \in \mathcal{T}_{m}, \; \sigma \in \Xi_{m}, \; m \in \mathcal{M}(\mathcal{T}).
    \end{align}
    Next, we develop the following piecewise $H^1$ estimate by using standard $L_{\infty}$ error estimates (cf.~\cite{brenner2008mathematical}, Chapter 4), in conjunction with Eqs.~\eqref{jump_one}, \eqref{jump_two}, and \eqref{diam_criterion}
    \begin{align}
        \left| \bm{v} - E \bm{v} \right|^{2}_{H^{1}(\Omega, \mathcal{T})} &\lesssim \sum_{T \in \mathcal{T}} \left(\text{diam} \, T\right) \left( \sum_{p \in \mathcal{V}(T)} \left| \bm{v}_{T} (p) - E\bm{v}(p) \right|^{2} + \sum_{m \in \mathcal{M}(T)} \left| \bm{v}_{T} (m) - E\bm{v}(m) \right|^{2} \right) \label{h1_estimate} \\[1.5ex]
        \nonumber & \lesssim \sum_{\sigma \in S(\mathcal{T},\Omega)} (\text{diam} \, \sigma) \left( \sum_{p \in \mathcal{V}(\sigma)} \left| \llbracket \bm{v} \rrbracket_{\sigma} (p) \right|^{2} + \sum_{m \in \mathcal{M}(\sigma)} \left| \llbracket \bm{v} \rrbracket_{\sigma} (m) \right|^{2} \right).
    \end{align}
    In addition, we develop the following $H^1$ estimate based on the triangle inequality, root-mean-square-arithmetic-mean inequality, Eqs.~\eqref{korn_seminorm_cont}, \eqref{seminorm_bound}, and \eqref{h1_estimate}, and Eq.~\eqref{grad_bound} from Lemma~\ref{grad_lemma} in Appendix A
    \begin{align*}
        \left| \bm{v} \right|^{2}_{H^{1}(\Omega,\mathcal{T})} &\leq \left| E \bm{v} \right|^{2}_{H^{1}(\Omega)} + \left| \bm{v} - E \bm{v} \right|^{2}_{H^{1}(\Omega, \mathcal{T})} \\[1.5ex]
        & \lesssim \left\| \bm{\epsilon}(E\bm{v}) - \frac{1}{3} \text{div} (E\bm{v}) \mathbb{I} \right\|^{2}_{L_2(\Omega) \times L_2(\Omega)} + \left( \bm{\Phi}(E\bm{v}) \right)^{2} + \left| \bm{v} - E \bm{v} \right|^{2}_{H^{1}(\Omega, \mathcal{T})} \\[1.5ex]
        &\lesssim \left\| \bm{\epsilon}_{\mathcal{T}}(\bm{v}) - \frac{1}{3} \text{div}_{\mathcal{T}} (\bm{v}) \mathbb{I} \right\|^{2}_{L_2(\Omega) \times L_2(\Omega)} + \left\| \bm{\epsilon}_{\mathcal{T}}(\bm{v} - E\bm{v}) - \frac{1}{3} \text{div}_{\mathcal{T}} (\bm{v} - E\bm{v}) \mathbb{I} \right\|^{2}_{L_2(\Omega) \times L_2(\Omega)} \\[1.5ex]
        &+ \left( \bm{\Phi}(\bm{v}) \right)^{2} + \left( \bm{\Phi}(\bm{v} - E\bm{v}) \right)^{2} + \left| \bm{v} - E \bm{v} \right|^{2}_{H^{1}(\Omega, \mathcal{T})} \\[1.5ex]
        &\lesssim \left\| \bm{\epsilon}_{\mathcal{T}}(\bm{v}) - \frac{1}{3} \text{div}_{\mathcal{T}} (\bm{v}) \mathbb{I} \right\|^{2}_{L_2(\Omega) \times L_2(\Omega)} + \left( \bm{\Phi}(\bm{v}) \right)^{2} + \left( \bm{\Phi}(\bm{v} - E\bm{v}) \right)^{2} + \left| \bm{v} - E \bm{v} \right|^{2}_{H^{1}(\Omega, \mathcal{T})} \\[1.5ex]
        & \lesssim \left\| \bm{\epsilon}_{\mathcal{T}}(\bm{v}) - \frac{1}{3} \text{div}_{\mathcal{T}} (\bm{v}) \mathbb{I} \right\|^{2}_{L_2(\Omega) \times L_2(\Omega)} + \left( \bm{\Phi}(\bm{v}) \right)^{2} \\[1.5ex]
        & + \sum_{\sigma \in S(\mathcal{T},\Omega)} (\text{diam} \, \sigma) \left( \sum_{p \in \mathcal{V}(\sigma)} \left| \llbracket \bm{v} \rrbracket_{\sigma} (p) \right|^{2} + \sum_{m \in \mathcal{M}(\sigma)} \left| \llbracket \bm{v} \rrbracket_{\sigma} (m) \right|^{2} \right).
    \end{align*}
    The last line contains the desired result (Eq.~\eqref{v_estimate}). 
\end{proof}

In the following examples, we will demonstrate that the inequality in Eq.~\eqref{seminorm_bound} holds for the seminorms in Eqs.~\eqref{seminorm_def_one} and \eqref{seminorm_def_two}.

\begin{example}
    For the seminorm in Eq.~\eqref{seminorm_def_one}, we have that
    \begin{align*}
        \left(\bm{\Phi}_{1}(\bm{v}-E\bm{v})\right)^{2} & = \left\| G(\bm{v} - E \bm{v}) \right\|_{L_2(\Omega)}^{2} \\[1.5ex]
        &\leq C \left\| \bm{v} - E \bm{v} \right\|_{L_2(\Omega)}^{2} \\[1.5ex]
        & \lesssim \sum_{T \in \mathcal{T}} \left(\text{diam} \, T \right)^{3} \left( \sum_{p\in\mathcal{V}(T)} \left| \bm{v}_{T}(p) - E \bm{v}(p) \right|^{2} + \sum_{m\in\mathcal{M}(T)} \left| \bm{v}_{T}(m) - E \bm{v}(m) \right|^{2} \right) \\[1.5ex]
        &\lesssim \sum_{\sigma \in S(\mathcal{T},\Omega)} (\text{diam} \, \sigma)^{3} \left( \sum_{p \in \mathcal{V}(\sigma)} \left| \llbracket \bm{v} \rrbracket_{\sigma} (p) \right|^{2} + \sum_{m \in \mathcal{M}(\sigma)} \left| \llbracket \bm{v} \rrbracket_{\sigma} (m) \right|^{2} \right),
    \end{align*}
    where we have used the boundedness of the projection operator (Eq.~\eqref{ortho_bound}), a standard $L_{\infty}$ estimate~\cite{brenner2008mathematical}, and Eqs.~\eqref{jump_one}, \eqref{jump_two}, and \eqref{diam_criterion}. 
\end{example}

\begin{example}
    For the seminorm in Eq.~\eqref{seminorm_def_two}, we have that
    \begin{align*}
        \left(\bm{\Phi}_{2}(\bm{v}-E\bm{v})\right)^{2} &= \sup_{ \substack{\bm{m} \in \mathbf{CK}(\Omega) \\
        \left\| \bm{m} \right\|_{L_2(\partial \Omega)} = 1, \int_{\partial \Omega} \bm{m} \, ds = \bm{0}}} \left( \int_{\partial \Omega} \left( \bm{v} - E \bm{v} \right) \cdot \bm{m} \, ds \right)^{2} \\[1.5ex]
        &\leq C \left\| \bm{v} - E \bm{v} \right\|_{L_2(\partial\Omega)}^{2}
        \\[1.5ex]
        & \lesssim \sum_{\substack{T \in \mathcal{T} \\ \partial T \cap \partial \Omega \neq \emptyset}} \left(\text{diam} \, T \right)^{2} \left( \sum_{p\in\mathcal{V}(T)} \left| \bm{v}_{T}(p) - E \bm{v}(p) \right|^{2} + \sum_{m\in\mathcal{M}(T)} \left| \bm{v}_{T}(m) - E \bm{v}(m) \right|^{2} \right) \\[1.5ex]
        &\lesssim \sum_{\sigma \in S(\mathcal{T},\Omega)} (\text{diam} \, \sigma)^{2} \left( \sum_{p \in \mathcal{V}(\sigma)} \left| \llbracket \bm{v} \rrbracket_{\sigma} (p) \right|^{2} + \sum_{m \in \mathcal{M}(\sigma)} \left| \llbracket \bm{v} \rrbracket_{\sigma} (m) \right|^{2} \right),
    \end{align*}
    where we have used H{\"o}lder's inequality, a standard $L_{\infty}$ estimate~\cite{brenner2008mathematical}, a discrete trace inequality, a root-mean-square-arithmetic-mean inequality, and Eqs.~\eqref{jump_one}, \eqref{jump_two}, and \eqref{diam_criterion}. 
\end{example}

In order to provide an appropriate context for our discussion of the final theorem, we will introduce the following interpolation operator for all $\bm{v} \in \left[H^{2}(T)\right]^{3}$
\begin{align}
    &\left| \int_{T} \left(\bm{v} - \Pi_{T} \bm{v} \right) dx \right| = 0, \qquad \left| \int_{T} \text{div} \left(\bm{v} - \Pi_{T} \bm{v} \right) dx \right| = 0, \label{interp_def} \\[1.5ex]
    \nonumber &\left| \int_{T} \text{curl} \left( \bm{v} - \Pi_{T} \bm{v} \right) dx \right| = 0, \qquad \left| \int_{T} \text{curl} \left( \text{curl} \left( \bm{v} - \Pi_{T} \bm{v} \right) \right) dx \right| = 0.
\end{align}
This interpolation operator is noteworthy because it projects $\bm{v}$ on to the kernel space $\mathbf{CK} (T)$. In fact, it can be shown that for all $\bm{m} \in \mathbf{CK}(T)$
\begin{align*}
    & \left| \int_{T} \bm{m} \, dx \right| = 0, \qquad \left| \int_{T} \text{div} (\bm{m}) \, dx \right| = 0, \\[1.5ex]
    &\left| \int_{T} \text{curl}(\bm{m}) \, dx \right| = 0, \qquad \left| \int_{T} \text{curl}(\text{curl} (\bm{m})) \, dx \right| = 0, \qquad \Longleftrightarrow \qquad \bm{m} = \bm{0}.
\end{align*}
The proof of the above statement appears in Appendix B, Lemma~\ref{projection_lemma}. Using the element-wise definition of the interpolation operator above, we can construct a global interpolation operator
\begin{align*}
    &\Pi : \left[H^{2}(\Omega,\mathcal{T})\right]^{3} \rightarrow V_{\mathcal{T}}, \\[1.5ex]
    & \left(\Pi \bm{v}\right) |_{T} = \Pi_{T} \bm{v}_{T}, \qquad \forall T \in \mathcal{T}. 
\end{align*}
In addition, upon combining Eq.~\eqref{interp_def} with Eq.~\eqref{local_korn}, (the latter of which is derived in Appendix C, Theorem~\ref{local_korn_lemma}), we obtain
\begin{align}
    \left| \bm{v} - \Pi_{T} \bm{v} \right|_{H^{1}(T)} &\lesssim \left\| \bm{\epsilon} (\bm{v} - \Pi_{T} \bm{v}) - \frac{1}{3} \text{div}(\bm{v} - \Pi_{T} \bm{v}) \mathbb{I} \right\|_{L_2(T) \times L_2(T)} \label{project_error_one} \\[1.5ex]
    \nonumber & = \left\| \bm{\epsilon} (\bm{v}) - \frac{1}{3} \text{div}(\bm{v}) \mathbb{I} \right\|_{L_2(T) \times L_2(T)},
\end{align}
for all $T \in \mathcal{T}$ and $\bm{v} \in \left[H^{2}(\Omega,\mathcal{T})\right]^{3}$. Finally, by a standard Poincar{\'e}-type inequality, we have that
\begin{align}
    \left\| \bm{v} - \Pi_{T} \bm{v} \right\|_{L_2(T)} \lesssim \left(\text{diam} \, T \right) \left| \bm{v} - \Pi_{T} \bm{v} \right|_{H^{1}(T)}, \label{project_error_two}
\end{align}
for all $T \in \mathcal{T}$ and $\bm{v} \in \left[H^{1}(\Omega,\mathcal{T})\right]^{3}$. We are now ready to state the final result.

\begin{theorem}
    Suppose that a seminorm $\bm{\Phi}: \left[H^{1}(\Omega, \mathcal{T})\right]^{3} \rightarrow \mathbb{R}$ satisfies the following condition
    \begin{align}
        \bm{\Phi} \left(\bm{v} - \Pi \bm{v} \right) \lesssim \left\| \bm{\epsilon}_{\mathcal{T}} (\bm{v}) - \frac{1}{3} \text{div}_{\mathcal{T}}(\bm{v}) \mathbb{I} \right\|_{L_2(\Omega) \times L_2(\Omega)}, \qquad \forall \bm{v} \in \left[ H^{2}(\Omega, \mathcal{T}) \right]^{3}, \label{seminorm_proj_bound}
    \end{align}
    where $\Pi : \left[ H^{2}(\Omega, \mathcal{T}) \right]^{3} \rightarrow V_{\mathcal{T}}$. Under these circumstances, the following inequality holds
    \begin{align}
        \left| \bm{v} \right|^{2}_{H^{1}(\Omega,\mathcal{T})} \leq \kappa\left(\theta_{\mathcal{T}} \right) \Bigg( &\left\| \bm{\epsilon}_{\mathcal{T}} (\bm{v}) - \frac{1}{3} \text{div}_{\mathcal{T}}(\bm{v}) \mathbb{I} \right\|_{L_2(\Omega) \times L_2(\Omega)}^{2} + \left(\bm{\Phi}(\bm{v})\right)^{2} \label{final_result} \\[1.5ex]
        \nonumber &+ \sum_{\sigma \in S(\mathcal{T}, \Omega)} \left(\text{diam} \, \sigma \right)^{-1} \left\| \pi_{\sigma} \llbracket \bm{v} \rrbracket_{\sigma} \right\|^{2}_{L_2(\sigma)} \Bigg), \qquad \forall \bm{v} \in \left[H^{2}(\Omega, \mathcal{T})\right]^{3},
    \end{align}
    where $\kappa: \mathbb{R}^{+} \rightarrow \mathbb{R}^{+}$ is a continuous function, which does not depend on $\mathcal{T}$. 
\end{theorem}

\begin{proof}
    We start by constructing the following $H^1$ estimate using the triangle inequality, the root-mean-square-arithmetic-mean inequality, Eq.~\eqref{kernel_condition}, Eq.~\eqref{v_estimate}, and Eq.~\eqref{project_error_one}
    \begin{align}
        \left| \bm{v} \right|_{H^{1}(\Omega, \mathcal{T})}^{2} &\leq \left| \bm{v} - \Pi \bm{v} \right|^{2}_{H^{1}(\Omega,\mathcal{T})} + \left| \Pi \bm{v} \right|^{2}_{H^{1}(\Omega, \mathcal{T})} \label{theorem_a} \\[1.5ex]
        \nonumber &\lesssim \left\| \bm{\epsilon}_{\mathcal{T}} (\bm{v}) - \frac{1}{3} \text{div}_{\mathcal{T}}(\bm{v}) \mathbb{I} \right\|_{L_2(\Omega) \times L_2(\Omega)}^{2} + \left(\bm{\Phi}(\Pi \bm{v}) \right)^{2} \\[1.5ex]
        \nonumber &+\sum_{\sigma \in S(\mathcal{T},\Omega)} (\text{diam} \, \sigma) \left( \sum_{p \in \mathcal{V}(\sigma)} \left| \llbracket \Pi \bm{v} \rrbracket_{\sigma} (p) \right|^{2} + \sum_{m \in \mathcal{M}(\sigma)} \left| \llbracket \Pi \bm{v} \rrbracket_{\sigma} (m) \right|^{2} \right).        
    \end{align}
    In addition, we can use the triangle inequality and Eq.~\eqref{seminorm_proj_bound} to obtain the following estimate
    \begin{align}
        \bm{\Phi}(\Pi \bm{v}) & \leq \bm{\Phi}(\bm{v} - \Pi \bm{v}) + \bm{\Phi}(\bm{v}) \label{theorem_b} \\[1.5ex]
        \nonumber &\lesssim \left\| \bm{\epsilon}_{\mathcal{T}} (\bm{v}) - \frac{1}{3} \text{div}_{\mathcal{T}}(\bm{v}) \mathbb{I} \right\|_{L_2(\Omega) \times L_2(\Omega)} + \bm{\Phi}(\bm{v}).
    \end{align}
    It remains for us to construct an upper bound for the jump terms on the RHS of Eq.~\eqref{theorem_a}. Towards this end, we re-introduce the quadratic projection operator
    \begin{align}
        \nonumber &\pi_{\sigma} : \left[L_{2}(\sigma)\right]^{3} \rightarrow \left[P_2(\sigma)\right]^{3}, \qquad \forall \sigma \in S(\mathcal{T}, \Omega).
    \end{align}
    Then, by the triangle inequality, the root-mean-square-arithmetic-mean inequality, the discrete trace inequality, and standard inverse estimates (see~\cite{brenner2008mathematical}, Chapter 4), we obtain
    \begin{align}
        \left| \llbracket \Pi \bm{v} \rrbracket_{\sigma} (p) \right|^{2} &= \left| \llbracket \Pi \bm{v} - \bm{v} + \bm{v} \rrbracket_{\sigma} (p) \right|^{2} \label{project_p} \\[1.5ex]
        \nonumber &= \left| \pi_{\sigma} \llbracket \Pi \bm{v} - \bm{v} + \bm{v} \rrbracket_{\sigma}(p) \right|^{2} \\[1.5ex]
        \nonumber &\lesssim \left| \pi_{\sigma} \llbracket \bm{v} - \Pi \bm{v} \rrbracket_{\sigma} (p) \right|^{2} + \left| \pi_{\sigma} \llbracket \bm{v} \rrbracket_{\sigma}(p) \right|^{2} \\[1.5ex]
        \nonumber &\lesssim \left(\text{diam} \, \sigma\right)^{-2} \left( \left\| \pi_{\sigma} \llbracket \bm{v} - \Pi \bm{v} \rrbracket_{\sigma} \right\|_{L_2(\sigma)}^{2} + \left\| \pi_{\sigma} \llbracket \bm{v} \rrbracket_{\sigma} \right\|_{L_2(\sigma)}^{2} \right).
    \end{align}
    Following a similar argument, we have that
    \begin{align}
        \left| \llbracket \Pi \bm{v} \rrbracket_{\sigma} (m) \right|^{2} \lesssim \left(\text{diam} \, \sigma\right)^{-2} \left( \left\| \pi_{\sigma} \llbracket \bm{v} - \Pi \bm{v} \rrbracket_{\sigma} \right\|_{L_2(\sigma)}^{2} + \left\| \pi_{\sigma} \llbracket \bm{v} \rrbracket_{\sigma} \right\|_{L_2(\sigma)}^{2} \right). \label{project_m}
    \end{align}
    Next, we can construct the following $L_2$ estimate by using the discrete trace inequality in conjunction with Eqs.~\eqref{project_error_one}, \eqref{project_error_two}, and \eqref{proj_ineq}
    \begin{align}
        \left\| \pi_{\sigma} \llbracket \bm{v} - \Pi \bm{v} \rrbracket_{\sigma} \right\|_{L_2(\sigma)}^{2} &\leq C \left\| \llbracket \bm{v} - \Pi \bm{v} \rrbracket_{\sigma} \right\|_{L_2(\sigma)}^{2} \label{theorem_c} \\[1.5ex]
        \nonumber & \lesssim \sum_{T \in \mathcal{T}_{\sigma}} \Bigg( \left(\text{diam} \, T \right) \left| \bm{v}_{T} - \Pi_{T} \bm{v}_{T} \right|_{H^{1}(T)}^{2} \\[1.5ex]
        \nonumber &+ \left(\text{diam} \, T\right)^{-1} \left\| \bm{v}_{T} - \Pi_{T} \bm{v}_{T} \right\|_{L_{2}(T)}^{2} \Bigg) \\[1.5ex]
        \nonumber &\lesssim \sum_{T \in \mathcal{T}_{\sigma}} \left(\text{diam} \, T \right) \left\| \bm{\epsilon}(\bm{v}_{T}) -\frac{1}{3} \text{div} (\bm{v}_{T}) \mathbb{I} \right\|_{L_2(T) \times L_2(T)}^{2},
    \end{align}
    where $\mathcal{T}_{\sigma}$ is the set of tetrahedra sharing the common side $\sigma$. Now, upon combining Eqs.~\eqref{diam_criterion}, \eqref{project_p}, \eqref{project_m}, and \eqref{theorem_c}, we obtain
    \begin{align}
        &\sum_{\sigma \in S(\mathcal{T},\Omega)} (\text{diam} \, \sigma) \left( \sum_{p \in \mathcal{V}(\sigma)} \left| \llbracket \Pi \bm{v} \rrbracket_{\sigma} (p) \right|^{2} + \sum_{m \in \mathcal{M}(\sigma)} \left| \llbracket \Pi \bm{v} \rrbracket_{\sigma} (m) \right|^{2} \right) \label{theorem_d} \\[1.5ex] 
        \nonumber & \lesssim \sum_{\sigma \in S(\mathcal{T},\Omega)} (\text{diam} \, \sigma)^{-1} \left\|\pi_{\sigma} \llbracket \bm{v} \rrbracket_{\sigma} \right\|_{L_2(\sigma)}^{2} + \left\| \bm{\epsilon}_{\mathcal{T}} (\bm{v}) - \frac{1}{3} \text{div}_{\mathcal{T}}(\bm{v}) \mathbb{I} \right\|_{L_2(\Omega) \times L_2(\Omega)}^{2}.
    \end{align}
    Lastly, on combining Eqs.~\eqref{theorem_a}, \eqref{theorem_b}, and \eqref{theorem_d} we obtain
    \begin{align*}
        \left| \bm{v} \right|_{H^{1}(\Omega, \mathcal{T})}^{2} &\lesssim \left\| \bm{\epsilon}_{\mathcal{T}} (\bm{v}) - \frac{1}{3} \text{div}_{\mathcal{T}}(\bm{v}) \mathbb{I} \right\|_{L_2(\Omega) \times L_2(\Omega)}^{2} + \left( \bm{\Phi} (\bm{v}) \right)^{2} \\[1.5ex]
        &+ \sum_{\sigma \in S(\mathcal{T},\Omega)} (\text{diam} \, \sigma)^{-1} \left\|\pi_{\sigma} \llbracket \bm{v} \rrbracket_{\sigma} \right\|_{L_2(\sigma)}^{2}.
    \end{align*}
    This equation is identical to the desired result (Eq.~\eqref{final_result}).
\end{proof}

In the following examples, we will demonstrate that the inequality in Eq.~\eqref{seminorm_proj_bound} holds for the seminorms in Eqs.~\eqref{seminorm_def_one} and \eqref{seminorm_def_two}.

\begin{example}
    For the seminorm in Eq.~\eqref{seminorm_def_one}, we have that
    \begin{align*}
    \left(\bm{\Phi}_{1}(\bm{v} - \Pi \bm{v} )\right)^{2} & = \sum_{T \in \mathcal{T}} \left\|G(\bm{v}_{T} - \Pi_{T} \bm{v}_{T}) \right\|_{L_{2}(T)}^{2} \\[1.5ex]
    &\leq C \sum_{T \in \mathcal{T}} \left\|\bm{v}_{T} - \Pi_{T} \bm{v}_{T} \right\|_{L_{2}(T)}^{2} \\[1.5ex]
    &\lesssim \sum_{T \in \mathcal{T}} \left(\text{diam} \, T\right)^{2} \left\| \bm{\epsilon} (\bm{v}_{T}) - \frac{1}{3} \text{div}(\bm{v}_{T}) \mathbb{I} \right\|_{L_2(T) \times L_2(T)}^{2} \\[1.5ex]
    &\lesssim \left\| \bm{\epsilon}_{\mathcal{T}} (\bm{v}) - \frac{1}{3} \text{div}_{\mathcal{T}} (\bm{v}) \mathbb{I} \right\|_{L_2(\Omega) \times L_2(\Omega)}^{2},
    \end{align*}
    where we have used Eqs.~\eqref{ortho_bound}, \eqref{project_error_one}, and \eqref{project_error_two}. 
\end{example}

\begin{example}
    For the seminorm in Eq.~\eqref{seminorm_def_two}, we have that
    \begin{align*}
    \left(\bm{\Phi}_{2}(\bm{v} - \Pi \bm{v} )\right)^{2} &= \sup_{ \substack{\bm{m} \in \mathbf{CK}(\Omega) \\
        \left\| \bm{m} \right\|_{L_2(\partial \Omega)} = 1, \int_{\partial \Omega} \bm{m} \, ds = \bm{0}}} \left( \int_{\partial \Omega} \left( \bm{v} - \Pi \bm{v} \right) \cdot \bm{m} \, ds \right)^{2} \\[1.5ex]
        &\leq C \left\| \bm{v} - \Pi \bm{v} \right\|_{L_2(\partial\Omega)}^{2} \\[1.5ex]
        &\lesssim \sum_{\substack{T \in \mathcal{T} \\ \partial T \cap \partial \Omega \neq \emptyset}} \left(\text{diam} \, T\right) \left\| \bm{\epsilon} (\bm{v}_{T}) - \frac{1}{3} \text{div}(\bm{v}_{T}) \mathbb{I} \right\|_{L_2(T) \times L_2(T)}^{2} \\[1.5ex]
        &\lesssim \left\| \bm{\epsilon}_{\mathcal{T}} (\bm{v}) - \frac{1}{3} \text{div}_{\mathcal{T}} (\bm{v}) \mathbb{I} \right\|_{L_2(\Omega) \times L_2(\Omega)}^{2},
    \end{align*}
    where we have used a discrete trace inequality, and Eqs.~\eqref{project_error_one} and \eqref{project_error_two}.
\end{example}

\section{Conclusion}

In this work, we introduced a class of discrete generalized Korn's inequalities. These inequalities are similar to standard Korn's inequalities, with the key caveat that they involve the trace-free symmetric gradient operator instead of the symmetric gradient operator. We have shown that the trace-free symmetric gradient operator requires a careful analysis, as it possesses a more complicated kernel space than the standard symmetric gradient operator. As a result, it is necessary for us to introduce a new quadratic mapping procedure, and a new projection operator in order to construct the corresponding discrete Korn's inequalities. To the best of our knowledge, this endeavor has been successful in generating a new class of discrete inequalities for piecewise $H^2$ vector fields in 3D. 

There are several interesting possibilities for future research. In particular, the extension of this work to 2D remains an open research question. In particular, the kernel space becomes infinite dimensional in this case, unless the vector field vanishes on the boundary~\cite{dain2006generalized}. In principle, such an extension is possible, although it is beyond the scope of the current work. In addition, we note that there are still opportunities to extend our analysis to higher dimensions, where more complicated derivative operators are present. This extension may be of particular interest for space-time finite element methods in 4D.


\appendix
\section{Seminorms and Generalized Korn's Inequalities}

In this section, our objective is to construct a pair of generalized Korn's inequalities which are augmented on the right hand side with seminorms. In preparation for constructing these inequalities, we will first construct an `inverse generalized Korn's inequality' that will facilitate our subsequent analysis.

\begin{lemma}
    For all $\bm{u} \in \left[H^{1} (\Omega) \right]^3$, the following upper bound on the trace-free symmetric gradient holds
    \begin{align}
        \left\| \bm{\epsilon}(\bm{u}) - \frac{1}{3} \text{div} (\bm{u}) \mathbb{I} \right\|_{L_2(\Omega) \times L_2(\Omega)} \leq C \left| \bm{u} \right|_{H^{1}(\Omega)}. \label{grad_bound_omega}
    \end{align}
    In addition, for all $\bm{u} \in \left[H^{1} (\Omega, \mathcal{T}) \right]^3$
    \begin{align}
        \left\| \bm{\epsilon}_{\mathcal{T}}(\bm{u}) - \frac{1}{3} \text{div}_{\mathcal{T}} (\bm{u}) \mathbb{I} \right\|_{L_2(\Omega) \times L_2(\Omega)} \leq C \left| \bm{u} \right|_{H^{1}(\Omega,\mathcal{T})}. \label{grad_bound}
    \end{align}
    \label{grad_lemma}
\end{lemma}

\begin{proof}
    We begin by considering a vector field, $\bm{v} \in \left[H^{1}(\Omega) \right]^{3}$. For this case, we have that
    \begin{align}
        &\left\| \bm{\epsilon}(\bm{v}) - \frac{1}{3} \text{div} (\bm{v}) \mathbb{I} \right\|_{L_2(\Omega) \times L_2(\Omega)} \label{grad_bound_local} \\[1.5ex]
        \nonumber &= \left\| \frac{1}{2}\left(\text{grad}(\bm{v}) + \text{grad}(\bm{v})^{t} \right) - \frac{1}{3} \text{div} (\bm{v}) \mathbb{I} \right\|_{L_2(\Omega) \times L_2(\Omega)} \\[1.5ex]
        \nonumber & \leq C \left( \left\| \text{grad}(\bm{v}) \right\|_{L_2(\Omega) \times L_2(\Omega)} + \left\| \text{grad}(\bm{v})^{t} \right\|_{L_2(\Omega) \times L_2(\Omega)} + \left\| \text{div}(\bm{v}) \mathbb{I} \right\|_{L_2(\Omega) \times L_2(\Omega)} \right) \\[1.5ex]
        \nonumber &\leq C \left( \left\| \text{grad}(\bm{v}) \right\|_{L_2(\Omega) \times L_2(\Omega)} + \left\| \text{div}(\bm{v}) \right\|_{L_2(\Omega)} \right),
    \end{align}
    where we have used the triangle inequality, and the fact that $\left\| \text{grad}(\bm{v}) \right\|_{L_2(\Omega) \times L_2(\Omega)} = \left\| \text{grad}(\bm{v})^{t} \right\|_{L_2(\Omega) \times L_2(\Omega)}$.
    
    In accordance with Lemma 3.34 of~\cite{john2016finite},
    \begin{align}
        \left\| \text{div}(\bm{v}) \right\|_{L_2(\Omega)} \leq C \left\| \text{grad}(\bm{v}) \right\|_{L_2(\Omega) \times L_2(\Omega)}.
    \end{align}
    Upon substituting this result into Eq.~\eqref{grad_bound_local}, we obtain
    \begin{align}
        \left\| \bm{\epsilon}(\bm{v}) - \frac{1}{3} \text{div} (\bm{v}) \mathbb{I} \right\|_{L_2(\Omega) \times L_2(\Omega)} \leq C \left\| \text{grad}(\bm{v}) \right\|_{L_2(\Omega) \times L_2(\Omega)}. \label{grad_bound_local_simp}
    \end{align}
    Now, upon replacing $\bm{v}$ with $\bm{u}$, we obtain the first desired result (Eq.~\eqref{grad_bound_omega}).
    
    Furthermore, upon replacing $\Omega$ with $T$ in Eq.~\eqref{grad_bound_local_simp}, squaring both sides, summing over all $T \in \mathcal{T}$, taking the square root, and letting
    \begin{align*}
        \bm{u}|_{T} = \bm{v},
    \end{align*}
    we obtain the second desired result (Eq.~\eqref{grad_bound}).
\end{proof}

We are now ready to introduce the first generalized Korn's inequality that applies to $H^1$ vector fields.

\begin{lemma}
        Suppose that $\bm{v} \in \left[H^{1}(\Omega)\right]^{3}$ and that a seminorm $\bm{\Phi} (\bm{v})$ satisfies the conditions in Eqs.~\eqref{phi_cond_one} and \eqref{phi_cond_two}. Under these circumstances, we have that
        \begin{align}
            \left| \bm{v} \right|_{H^{1} (\Omega)} \leq C_{\Omega} \left( \left\| \bm{\epsilon} \left(\bm{v}\right) - \frac{1}{3} \text{div} \left(\bm{v} \right) \mathbb{I} \right\|_{L_2(\Omega) \times L_2(\Omega)} + \bm{\Phi} \left(\bm{v}\right) \right), \qquad \forall \bm{v} \in \left[ H^1(\Omega) \right]^{3}.
            \label{korn_seminorm_cont_final}
        \end{align}
        %
        \label{seminorm_lemma}
\end{lemma}

\begin{proof}
    The proof follows from Lemma~\ref{grad_lemma} and the compact embedding of $H^1$ into~$L_2$, (cf.~\cite{adams2003sobolev}, Theorem 6.3). 
\end{proof}

Our next generalized Korn's inequality applies to vector fields in $\mathbb{E}(\Omega)$, which is a complete, normed vector space which lies between $[H^1(\Omega)]^{3}$ and $[H^2(\Omega)]^{3}$. In what follows, we will define this space, and then present the associated generalized Korn's inequality. 

\begin{definition}[The Space $\mathbb{E}(\Omega)$]
   We begin by defining a new norm
    \begin{align*}
        \left\| \bm{v} \right\|_{\mathbb{E}(\Omega)} = \left\| \bm{v}  \right\|_{H^{1}(\Omega)} + \left| \int_{\Omega} \text{curl}(\text{curl} ( \bm{v} )) \, dx \right|,
    \end{align*}
    for all $\bm{v} \in \mathbb{E}(\Omega)$. Here, we define $\mathbb{E}(\Omega)$ as a Banach space which is a subspace of $\left[H^{1}(\Omega)\right]^{3}$ such that
    \begin{align*}
        \mathbb{E}(\Omega) = \left\{ \bm{v} \in \left[H^{1} (\Omega) \right]^{3} : \left| \int_{\Omega} \text{curl}(\text{curl} ( \bm{v} )) \, dx \right| < \infty \right\}.   
    \end{align*}
    We note that constructing $\mathbb{E}(\Omega)$ is always possible, as we can simply start with an incomplete metric space $\mathbb{E}'(\Omega)$, and then develop its completion in order to obtain $\mathbb{E}(\Omega)$ in accordance with~\cite{kreyszig1991introductory}, Theorem 1.6-2. For example, we can think of $\mathbb{E}(\Omega)$ as the completion of $\left[C^{\infty}(\Omega)\right]^{3}$ with respect to the norm $\left\| \cdot \right\|_{\mathbb{E}(\Omega)}$.

    By inspection, we have that the following inclusion holds
    \begin{align*}
    \left[H^{2}(\Omega)\right]^{3} \subset \mathbb{E}(\Omega) \subset \left[H^{1}(\Omega)\right]^{3}.
    \end{align*}
    Further, we assume that the following compact embedding holds
    \begin{align*}
        \mathbb{E}(\Omega) \subset \subset \left[L_{2}(\Omega)\right]^{3}.
    \end{align*}
\end{definition}

\begin{definition}[Div-Curl Seminorm]
   Consider the seminorm $\bm{\Psi}$ which is well-defined for vectors $\bm{v} \in \mathbb{E}(\Omega)$,
    \begin{align}
        \bm{\Psi} (\bm{v}) = \left| \int_{\Omega} \text{div}(\bm{v}) \, dx \right| + \left| \int_{\Omega} \text{curl}(\bm{v}) \, dx \right| +  \left| \int_{\Omega} \text{curl}(\text{curl} ( \bm{v} )) \, dx \right|. \label{particular_seminorm}
    \end{align}
    This seminorm has the following properties
    \begin{align}
        \bm{\Psi} \left(\bm{v}\right) \leq C_{\bm{\Psi}} \left\| \bm{v} \right\|_{\mathbb{E}(\Omega)}, \qquad \forall \bm{v} \in \mathbb{E}(\Omega), \label{phi_cond_one_exp}
    \end{align}
    and
    \begin{align}
        \bm{\Psi} \left(\bm{m}\right) = 0 \qquad \text{and} \qquad \bm{m} \in \mathbf{CK}(\Omega) \qquad \Longleftrightarrow \qquad \bm{m} = \text{a constant vector}. \label{phi_cond_two_exp}
    \end{align}
    It immediately follows that
    \begin{align}
        \bm{\Psi} \left(\bm{v} + \bm{c} \right) = \bm{\Psi}(\bm{v}).
        \label{phi_cond_three_exp}
    \end{align}
\end{definition}

\begin{lemma}
        Suppose that $\bm{v} \in \mathbb{E}(\Omega)$ and that a particular seminorm  (Eq.~\eqref{particular_seminorm}) satisfies the conditions in Eqs.~\eqref{phi_cond_one_exp} and \eqref{phi_cond_two_exp}. Under these circumstances, we have that
        \begin{align}
            \left| \bm{v} \right|_{H^{1} (\Omega)} + \left| \int_{\Omega} \text{curl}(\text{curl} ( \bm{v} )) \, dx \right| \leq C_{\Omega} \left( \left\| \bm{\epsilon} \left(\bm{v}\right) - \frac{1}{3} \text{div} \left(\bm{v} \right) \mathbb{I} \right\|_{L_2(\Omega) \times L_2(\Omega)} + \bm{\Psi} \left(\bm{v}\right) \right),
           \label{curl_curl_korn_seminorm_cont_final}
            \\[1.5ex] \nonumber \forall \bm{v} \in \mathbb{E}(\Omega).
        \end{align}
        \label{curl_curl_seminorm_lemma}
\end{lemma}

\begin{proof}
    The proof is by contradiction. We assume that Eq.~\eqref{curl_curl_korn_seminorm_cont_final} is false. In this case, there must exist a sequence $\bm{v}_n \in \mathbb{E}(\Omega)$ which satisfies the following conditions
    \begin{align}
        \left| \bm{v}_n \right|_{H^{1}(\Omega)} + \left| \int_{\Omega} \text{curl}(\text{curl} ( \bm{v}_{n} )) \, dx \right|  = 1, \qquad \forall n \geq 1, \label{embed_one_new}
    \end{align}
    and 
    \begin{align}
        1 \geq n \left( \left\| \bm{\epsilon} \left(\bm{v}_{n} \right) - \frac{1}{3} \text{div} \left(\bm{v}_{n} \right) \mathbb{I} \right\|_{L_2(\Omega) \times L_2(\Omega)} + \bm{\Psi} \left(\bm{v}_{n} \right) \right). \label{embed_one_a_new}
    \end{align}
    We can further specify the sequence, by assuming that 
    \begin{align}
        \int_{\Omega} \bm{v}_{n} \, dx = \bm{0}, \qquad \forall n \geq 1, \label{embed_two_new}
    \end{align}
    in accordance with Eq.~\eqref{phi_cond_three_exp}.
    Next, Eqs.~\eqref{embed_one_new}, \eqref{embed_two_new}, and a standard Poincar{\'e}-Friedrichs inequality imply that $\left\{\bm{v}_{n} \right\}_{n=1}^{\infty}$ remains bounded in $\mathbb{E}(\Omega)$.
    
    Now, due to the compact embedding of $\mathbb{E}(\Omega)$ into $\left[L_2(\Omega)\right]^{3}$, we can consider a convergent subsequence in $\left[L_2(\Omega)\right]^{3}$. More precisely, we can assume there exists $\bm{v}_{\ast} \in \left[L_{2}(\Omega) \right]^{3}$ such that
    \begin{align}
        \lim_{n \rightarrow \infty} \left\| \bm{v}_{n} - \bm{v}_{\ast} \right\|_{L_2(\Omega)} = 0. \label{embed_three_new}
    \end{align}
    Combining this statement along with Eq.~\eqref{embed_one_a_new}, we have that
    \begin{align}
        \lim_{n\rightarrow\infty}  \left\| \bm{\epsilon} \left(\bm{v}_{n} \right) - \frac{1}{3} \text{div} \left(\bm{v}_{n} \right) \mathbb{I} \right\|_{L_2(\Omega) \times L_2(\Omega)} = 0, \qquad \lim_{n\rightarrow\infty} \bm{\Psi}(\bm{v}_{n}) = 0. \label{embed_four_new}
    \end{align}
    The limits in Eqs.~\eqref{embed_three_new} and \eqref{embed_four_new}, in conjunction with Korn's inequality (Eq.~\eqref{korn_basic_cont_full}) imply that $\left\{\bm{v}_{n} \right\}_{n=1}^{\infty}$ is a Cauchy sequence in $\mathbb{E}(\Omega)$. For the sake of clarity, we will further expand this argument below. First, let Eq.~\eqref{korn_basic_cont_full} be rewritten as follows
    \begin{align}
        &\left\| \bm{v} \right\|_{H^{1} (\Omega)} +         \left| \int_{\Omega} \text{curl}(\text{curl} ( \bm{v} )) \, dx \right| \label{curl_curl_korn} \\[1.5ex]
        \nonumber &\leq \left(C_{\Omega} + 1 \right) \left( \left\| \bm{\epsilon} \left(\bm{v}\right) - \frac{1}{3} \text{div} \left(\bm{v} \right) \mathbb{I} \right\|_{L_2(\Omega) \times L_2(\Omega)} + \left\| \bm{v} \right\|_{L_2(\Omega)} + \left| \int_{\Omega} \text{curl}(\text{curl} ( \bm{v} )) \, dx \right| \right).
    \end{align}
    Next, we can consider two sequence members $\bm{v}_{n_{j}}$ and $\bm{v}_{n_{k}}$, where $k > j$, and both $k$ and $j$ are large. Upon substituting these members into Eq.~\eqref{curl_curl_korn}, we obtain
    \begin{align*}
        &\left\| \bm{v}_{n_{j}} - \bm{v}_{n_{k}} \right\|_{H^{1} (\Omega)} +         \left| \int_{\Omega} \text{curl}(\text{curl} ( \bm{v}_{n_{j}} - \bm{v}_{n_{k}} )) \, dx \right|  \\[1.5ex]
        &\leq \left(C_{\Omega} + 1 \right) \Bigg( \left\| \bm{\epsilon} \left(\bm{v}_{n_{j}} - \bm{v}_{n_{k}}\right) - \frac{1}{3} \text{div} \left(\bm{v}_{n_{j}} - \bm{v}_{n_{k}} \right) \mathbb{I} \right\|_{L_2(\Omega) \times L_2(\Omega)} \\[1.5ex]
        & + \left\| \bm{v}_{n_{j}} - \bm{v}_{n_{k}} \right\|_{L_2(\Omega)} + \left| \int_{\Omega} \text{curl}(\text{curl} ( \bm{v}_{n_{j}} - \bm{v}_{n_{k}} )) \, dx \right| \Bigg).
    \end{align*}
    On using the triangle inequality, we obtain
    \begin{align*}
        &\left\| \bm{v}_{n_{j}} - \bm{v}_{n_{k}} \right\|_{H^{1} (\Omega)} +         \left| \int_{\Omega} \text{curl}(\text{curl} ( \bm{v}_{n_{j}} - \bm{v}_{n_{k}} )) \, dx \right|  \\[1.5ex]
        &\leq \left(C_{\Omega} + 1 \right) \Bigg( \left\| \bm{\epsilon} \left(\bm{v}_{n_{j}} \right) - \frac{1}{3} \text{div} \left(\bm{v}_{n_{j}}  \right) \mathbb{I} \right\|_{L_2(\Omega) \times L_2(\Omega)} + \left\| \bm{\epsilon} \left( \bm{v}_{n_{k}}\right) - \frac{1}{3} \text{div} \left( \bm{v}_{n_{k}} \right) \mathbb{I} \right\|_{L_2(\Omega) \times L_2(\Omega)} \\[1.5ex]
        & + \left\| \bm{v}_{n_{j}} - \bm{v}_{n_{k}} \right\|_{L_2(\Omega)} + \left| \int_{\Omega} \text{curl}(\text{curl} ( \bm{v}_{n_{j}} )) \, dx \right| + \left| \int_{\Omega} \text{curl}(\text{curl} ( \bm{v}_{n_{k}} )) \, dx \right| \Bigg).
    \end{align*}
    In accordance with Eq.~\eqref{embed_one_a_new}, the RHS can be rewritten as follows
    \begin{align}
        &\left\| \bm{v}_{n_{j}} - \bm{v}_{n_{k}} \right\|_{H^{1} (\Omega)} +         \left| \int_{\Omega} \text{curl}(\text{curl} ( \bm{v}_{n_{j}} - \bm{v}_{n_{k}} )) \, dx \right| \label{cauchy_proof}  \\[1.5ex]
        \nonumber &\leq \left(C_{\Omega} + 1 \right) \Bigg( \frac{1}{n_j} + \frac{1}{n_k} + \left\| \bm{v}_{n_{j}} - \bm{v}_{n_{k}} \right\|_{L_2(\Omega)}  \Bigg).
    \end{align}
    Evidently, the sequence is Cauchy in $\mathbb{E}(\Omega)$, as $n_j \rightarrow \infty$, $n_k \rightarrow \infty$, and Eq.~\eqref{embed_three_new} holds.
    
    In addition, the limit of $\left\{ \bm{v}_n \right\}$ is $\bm{v}_{\ast} \in \mathbb{E}(\Omega)$ because $\mathbb{E}(\Omega)$ is a Banach space, and naturally
    \begin{align}
        \lim_{n \rightarrow \infty} \left\| \bm{v}_{n} - \bm{v}_{\ast} \right\|_{\mathbb{E}(\Omega)} = 0. \label{embed_five_new}
    \end{align}
    Having established the properties of our sequence in $\mathbb{E}(\Omega)$, we now turn our attention to the seminorm $\bm{\Psi}$. In accordance with Eq.~\eqref{phi_cond_one_exp} and the reverse triangle inequality, we have
    \begin{align}
        \left| \bm{\Psi}(\bm{v}_{n}) - \bm{\Psi}(\bm{v}_{\ast}) \right| \leq \bm{\Psi}(\bm{v}_{n} - \bm{v}_{\ast}) \leq C_{\bm{\Psi}} \left\| \bm{v}_{n} - \bm{v}_{\ast} \right\|_{\mathbb{E}(\Omega)}. \label{embed_six_new}
    \end{align}
    Upon combining Eqs.~\eqref{embed_four_new}, \eqref{embed_five_new}, and \eqref{embed_six_new}, we obtain
    \begin{align}
        \bm{\Psi}(\bm{v}_{\ast}) = 0. \label{embed_seven_new}
    \end{align}
    In a similar fashion, we have that
    \begin{align}
        &\left| \left\| \bm{\epsilon} \left(\bm{v}_{n} \right) - \frac{1}{3} \text{div} \left(\bm{v}_{n} \right) \mathbb{I} \right\|_{L_2(\Omega) \times L_2(\Omega)} - \left\| \bm{\epsilon} \left(\bm{v}_{\ast} \right) - \frac{1}{3} \text{div} \left(\bm{v}_{\ast} \right) \mathbb{I} \right\|_{L_2(\Omega) \times L_2(\Omega)} \right| \label{embed_eight_new} \\[1.5ex]
        \nonumber &\leq \left\| \bm{\epsilon} \left(\bm{v}_{n} - \bm{v}_{\ast} \right) - \frac{1}{3} \text{div} \left(\bm{v}_{n} - \bm{v}_{\ast} \right) \mathbb{I} \right\|_{L_2(\Omega) \times L_2(\Omega)} \leq C \left\|\bm{v}_{n} - \bm{v}_{\ast} \right\|_{H^{1}(\Omega)} \leq C \left\|\bm{v}_{n} - \bm{v}_{\ast} \right\|_{\mathbb{E}(\Omega)}, 
       \end{align}     
        where we have used Eq.~\eqref{grad_bound_omega} from Lemma~\ref{grad_lemma} on the last line. Upon combining Eqs.~\eqref{embed_four_new}, \eqref{embed_five_new}, and \eqref{embed_eight_new}, we obtain
        \begin{align}
            \bm{\epsilon} \left(\bm{v}_{\ast} \right) - \frac{1}{3} \text{div} \left(\bm{v}_{\ast} \right) \mathbb{I} = \bm{0}. \label{embed_nine_new}
        \end{align}
        Finally, based on Eqs.~\eqref{kernel_condition}, \eqref{phi_cond_two_exp}, \eqref{embed_seven_new}, and \eqref{embed_nine_new}, we conclude that $\bm{v}_{\ast}$ belongs to the kernel space $\mathbf{CK}(\Omega)$, and is a constant vector. However, this immediately contradicts Eqs.~\eqref{embed_one_new} and \eqref{embed_five_new}. 
\end{proof}

\section{Projection Operator}

In this section, we introduce an element-wise projection operator that projects H2 vector fields on to the kernel space $\mathbf{CK}(T)$, which itself is contained in the space of quadratic polynomial vector fields.

\begin{lemma}
        Consider the projection operator $\Pi_{T} : \left[ H^{2}(T) \right]^{3} \rightarrow \left[P_2(T)\right]^{3}$, which is defined such that
        \begin{align}
        &\left| \int_{T} \left(\bm{v} - \Pi_{T} \bm{v} \right) dx \right| = 0, \qquad \left| \int_{T} \text{div} \left(\bm{v} - \Pi_{T} \bm{v} \right) dx \right| = 0, \label{interp_repeat} \\[1.5ex]
        \nonumber &\left| \int_{T} \text{curl} \left( \bm{v} - \Pi_{T} \bm{v} \right) dx \right| = 0, \qquad \left| \int_{T} \text{curl} \left( \text{curl} \left( \bm{v} - \Pi_{T} \bm{v} \right) \right) dx \right| = 0.
        \end{align}
        This operator is well-posed, and projects $\bm{v} \in \left[ H^{2}(T) \right]^{3}$ on to the kernel space $\mathbf{CK}(T)$. 
        \label{projection_lemma}
\end{lemma}

\begin{proof}
    Let us introduce the function $\bm{w} \in \mathbf{CK}(T)$, where
    \begin{align}
        \bm{w} = 2 \left(\bm{a} \cdot \bm{x}\right) \bm{x} - \left| \bm{x} \right|^2 \bm{a} + \rho \bm{x} + \bm{Q} \bm{x} + \bm{b}, \label{expanded_w}
    \end{align}
    and where
    \begin{align}
        \bm{a} = 
        \begin{bmatrix}
        a_1 \quad a_2 \quad a_3
        \end{bmatrix}^T, \qquad 
        \bm{b} = 
        \begin{bmatrix}
        b_1 \quad b_2 \quad b_3
        \end{bmatrix}^T, 
        \qquad
        \bm{Q} = 
        \begin{bmatrix}
        0 & q_1 & q_2 \\[1.5ex]
        -q_1 & 0 & q_3 \\[1.5ex]
        -q_2 & -q_3 & 0
        \end{bmatrix}.
    \end{align}
    In order to ensure well-posedness of $\Pi_{T}$, and that the range of its projection is the appropriate kernel space, we require that $\bm{w} = \bm{0}$ when
    \begin{align}
        &\int_{T} \bm{w} \, dx=\bm{0}, \label{cond_one} \qquad \int_{T} \text{div}( \bm{w}) \, dx=0,  \\[1.5ex]
        &\int_{T} \text{curl}(\bm{w}) \, dx=\bm{0}, \label{cond_three} \qquad \int_{T} \text{curl} (\text{curl} (\bm{w}) ) \, dx=\bm{0}. 
    \end{align}
    Fortunately, this relationship is relatively easy to demonstrate. We start by substituting Eq.~\eqref{expanded_w} into the rightmost part of Eq.~\eqref{cond_three} as follows
    \begin{align*}
        \int_{T} \text{curl} (\text{curl} (\bm{w})) \, dx = 8 \int_{T} \bm{a} \, dx = \bm{0}.
    \end{align*}
    This immediately implies that $\bm{a} = \bm{0}$. Next, we substitute Eq.~\eqref{expanded_w} into the leftmost part of Eq.~\eqref{cond_three} in order to obtain
    \begin{align*}
        \int_{T} \text{curl}( \bm{w}) \, dx = 2 \int_{T} \begin{bmatrix}
        -\left(q_3 + 2 a_3 x_2 - 2 a_2 x_3 \right) \\[1.5ex]
        \left(q_2 + 2 a_3 x_1 - 2 a_1 x_3 \right) \\[1.5ex]
        -\left(q_1 + 2 a_2 x_1 - 2 a_1 x_2 \right)
        \end{bmatrix}
        \, dx = 2 \int_{T} \begin{bmatrix}
        -q_3  \\[1.5ex]
        q_2   \\[1.5ex]
        -q_1 
        \end{bmatrix}
        \, dx = \bm{0}.
    \end{align*}
    This immediately implies that $\bm{Q} = \bm{0}$. Next, we substitute Eq.~\eqref{expanded_w} into the rightmost part of Eq.~\eqref{cond_one}, such that
    \begin{align*}
        \int_{T} \text{div} ( \bm{w}) \, dx = \int_{T} \left( 6 \left(a_1 x_1 + a_2 x_2 + a_3 x_3 \right) + 3 \rho \right) \, dx = 3 \int_{T} \rho \, dx = 0.
    \end{align*}
    This immediately implies that $\rho = 0$. Finally, we substitute Eq.~\eqref{expanded_w} into the leftmost part of Eq.~\eqref{cond_one} as follows
    \begin{align*}
        \int_{T} \bm{w} \, dx &= \int_{T} 
        \begin{bmatrix}
        b_1 + q_1 x_2 + q_2 x_3 + a_1 \left(x_1^2 - x_2^2 - x_3^2\right) + 
      x_1 \left(2 a_2 x_2 + 2 a_3 x_3 + \rho \right) \\[1.5ex]
     b_2 - q_1 x_1 + q_3 x_3 - a_2 \left(x_1^2 - x_2^2 + x_3^2 \right) + 
      x_2 \left(2 a_1 x_1 + 2 a_3 x_3 + \rho \right) \\[1.5ex] 
     b_3 - q_2 x_1 - q_3 x_2 - a_3 \left(x_1^2 + x_2^2 - x_3^2 \right) + 
      x_3 \left(2 a_1 x_1 + 2 a_2 x_2 + \rho \right)
        \end{bmatrix}
        \, dx \\[1.5ex]
        & = \int_{T} \bm{b} \, dx = \bm{0}.
    \end{align*}
    This immediately implies that $\bm{b} = \bm{0}$. Therefore, since $\rho = 0$, $\bm{a} = \bm{b} = \bm{0}$, and $\bm{Q} = \bm{0}$, by inspection we have that $\bm{w} = \bm{0}$.
\end{proof}

\section{Generalized Korn's Inequality for Tetrahedra}

In this section, we derive a generalized Korn's inequality for tetrahedral meshes. Prior to introducing this inequality, we establish several useful definitions along with an important lemma. Thereafter, we summarize the principal result of the section in a theorem.

\begin{definition}[Reformulation of the H2 Seminorm]
    Let us begin by redefining the H2 seminorm for vectors $\bm{v} \in [H^{2}(\Omega)]^{3}$,
\begin{align*}
    \vertiii{\bm{v}}_{H^{2}(\Omega)} =   \left( \sum_{i=1}^{3} \int_{\Omega} \left| \text{hess}\left( v_i \right) \right|^{2} dx \right)^{1/2},
\end{align*}
where the `$\text{hess}$' operator is the second-derivative Hessian operator which acts on a generic scalar $\phi \in H^{2}(\Omega)$ and returns a $3\times 3$ symmetric matrix as follows
\begin{align*}
    \left[\text{hess} (\phi)\right]_{ij} = \frac{\partial^{2} \phi}{\partial x_i\partial x_j}.
\end{align*}
%
\end{definition}

\begin{definition}[Div-Curl Seminorm with Scaling]
   Consider the seminorm $\bm{\Psi}_{\ell}$ for vectors $\bm{v} \in [H^{2}(\Omega)]^{3}$,
    \begin{align}
        \bm{\Psi}_{\ell} (\bm{v}) = \ell^{a} \left( \left| \int_{\Omega} \text{div}(\bm{v}) \, dx \right| + \left| \int_{\Omega} \text{curl}(\bm{v}) \, dx \right| \right) + \ell^{b} \left| \int_{\Omega} \text{curl}(\text{curl} ( \bm{v} )) \, dx \right|, \label{particular_seminorm_redux}
    \end{align}
    where $\ell$ is a generic length scale, and $a$ and $b$ are generic constants. In 3-dimensional space, it is common practice to set $a = -3/2$ and $b = -1/2$ for scaling purposes, although this is not strictly necessary.
\end{definition}




We are now ready to introduce an important lemma.

\begin{lemma}
    Consider a vector field defined on the reference element $\widehat{T}$ such that $\widehat{\bm{v}} \in [ H^{2}(\widehat{T})]^{3}$ and
    \begin{align}
        \left\| \widehat{\bm{v}} \right\|_{H^{2}(\widehat{T})} = 1.
    \end{align}
    Next, suppose that an invertible linear mapping exists between $\widehat{T}$ and any generic tetrahedron $T$
    \begin{align}
        \bm{\alpha} (\widehat{\bm{x}}) = \bm{B} \widehat{\bm{x}},
    \end{align}
    where $\bm{B}$ belongs to the Lie group of non-singular matrices, $GL(3)$. Under these circumstances, the following function is continuous on $GL(3)$
    \begin{align}
        \Lambda_{\widehat{\bm{v}}}(\bm{B}) = \frac{\left| \widehat{\bm{v}} \circ \bm{\alpha}^{-1} \right|_{H^{1}(T)}}{ \left\| \bm{\epsilon} \left(\widehat{\bm{v}} \circ \bm{\alpha}^{-1}\right) - \frac{1}{3} \text{div}(\widehat{\bm{v}} \circ \bm{\alpha}^{-1} ) \mathbb{I} \right\|_{L_{2}(T) \times L_{2}(T)} +\bm{\Psi}_{\ell} \left(\widehat{\bm{v}} \circ \bm{\alpha}^{-1} \right)}. \label{particular_function}
    \end{align}
    \label{continuity_lemma}
\end{lemma}

\begin{proof}
    In order to prove continuity of the function in Eq.~\eqref{particular_function}, our goal is to construct an upper bound for the following quantity
    \begin{align}
        \nonumber \left|\Lambda_{\widehat{\bm{v}}}(\bm{B}_1) - \Lambda_{\widehat{\bm{v}}}(\bm{B}_2) \right| &= \Bigg| \frac{\left| \widehat{\bm{v}} \circ \bm{\alpha}_{1}^{-1} \right|_{H^{1}(T_1)}}{ \left\| \bm{\epsilon} \left(\widehat{\bm{v}} \circ \bm{\alpha}_{1}^{-1}\right) - \frac{1}{3} \text{div}(\widehat{\bm{v}} \circ \bm{\alpha}_{1}^{-1} ) \mathbb{I} \right\|_{L_{2}(T_1) \times L_{2}(T_1)} +\bm{\Psi}_{\ell} \left(\widehat{\bm{v}} \circ \bm{\alpha}_{1}^{-1} \right)} \\[1.5ex]
        &- \frac{\left| \widehat{\bm{v}} \circ \bm{\alpha}_{2}^{-1} \right|_{H^{1}(T_2)}}{ \left\| \bm{\epsilon} \left(\widehat{\bm{v}} \circ \bm{\alpha}_{2}^{-1}\right) - \frac{1}{3} \text{div}(\widehat{\bm{v}} \circ \bm{\alpha}_{2}^{-1} ) \mathbb{I} \right\|_{L_{2}(T_2) \times L_{2}(T_2)} +\bm{\Psi}_{\ell} \left(\widehat{\bm{v}} \circ \bm{\alpha}_{2}^{-1} \right)} \Bigg|, \label{quotient_diff}
    \end{align}
    in terms of $\left\|\bm{B}_{1}^{-1} - \bm{B}_{2}^{-1} \right\|_{F}$, where $T_1$ and $T_2$ are generic tetrahedra, (figure~\ref{continuity_fig} illustrates the geometric configuration). 
    \begin{figure}[h!]
        \includegraphics[width=8cm]{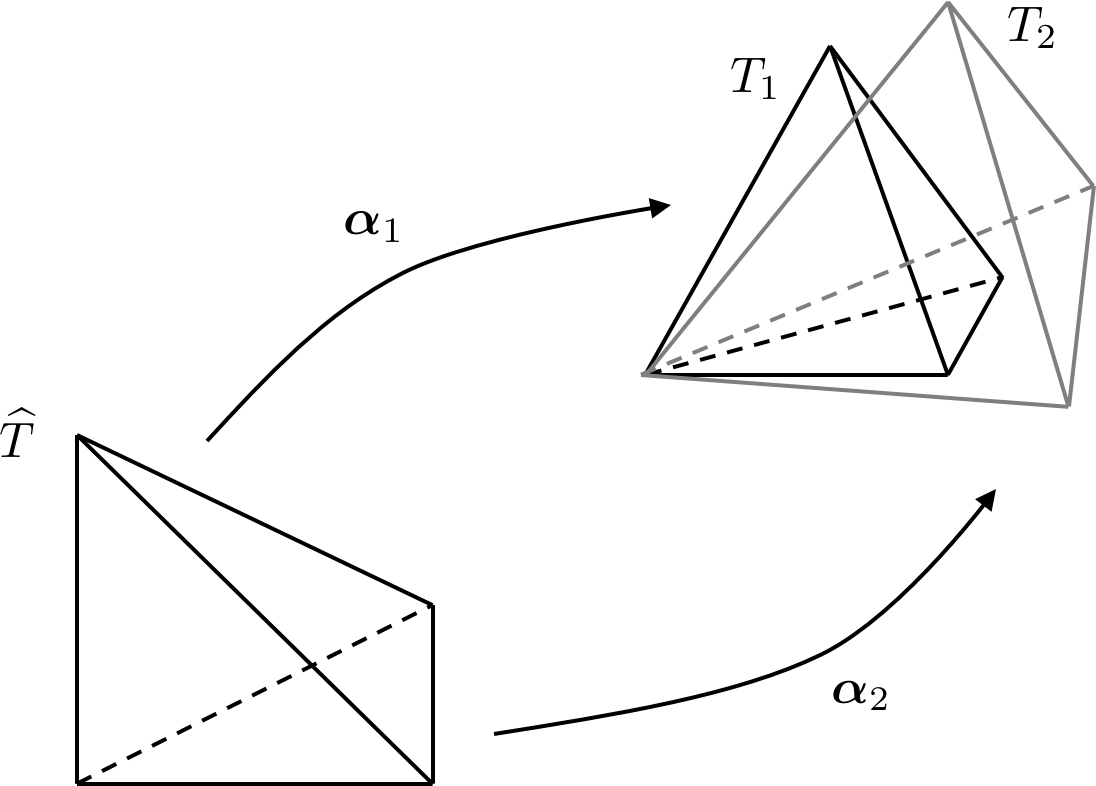}
        \caption{Mapping between reference tetrahedra $\widehat{T}$ and arbitrary tetrahedra $T_1$ and $T_2$.}
        \label{continuity_fig}
    \end{figure}
    Towards this end, we define the following surrogates for the quotient denominators in Eq.~\eqref{quotient_diff}
    \begin{align*}
        \mathcal{S}_1 &= \left\| \bm{\epsilon} \left(\widehat{\bm{v}} \circ \bm{\alpha}_{1}^{-1}\right) - \frac{1}{3} \text{div}(\widehat{\bm{v}} \circ \bm{\alpha}_{1}^{-1} ) \mathbb{I} \right\|_{L_{2}(T_1) \times L_{2}(T_1)} +\bm{\Psi}_{\ell} \left(\widehat{\bm{v}} \circ \bm{\alpha}_{1}^{-1} \right), \\[1.5ex]
        \mathcal{S}_{2} &= \left\| \bm{\epsilon} \left(\widehat{\bm{v}} \circ \bm{\alpha}_{2}^{-1}\right) - \frac{1}{3} \text{div}(\widehat{\bm{v}} \circ \bm{\alpha}_{2}^{-1} ) \mathbb{I} \right\|_{L_{2}(T_2) \times L_{2}(T_2)} +\bm{\Psi}_{\ell} \left(\widehat{\bm{v}} \circ \bm{\alpha}_{2}^{-1} \right).
    \end{align*}
    Upon substituting the expressions above into Eq.~\eqref{quotient_diff}, and manipulating the result, we obtain
    \begin{align}
         \left|\Lambda_{\widehat{\bm{v}}}(\bm{B}_1) - \Lambda_{\widehat{\bm{v}}}(\bm{B}_2) \right| &= \left| \frac{\left| \widehat{\bm{v}} \circ \bm{\alpha}_{1}^{-1} \right|_{H^{1}(T_1)} \mathcal{S}_{2} - \left| \widehat{\bm{v}} \circ \bm{\alpha}_{2}^{-1} \right|_{H^{1}(T_2)} \mathcal{S}_{1}}{\mathcal{S}_1 \mathcal{S}_{2}} \right|. \label{quotient_diff_one}
    \end{align}
    Next, we observe that the following identity holds for generic real numbers $c$, $d$, $e$, and~$f$
    \begin{align*}
        \left| cd-ef\right| &\leq |c| |d-f| + |f| |c - e| \leq \left(|c| + |f| \right) \left(|d-f| + |c-e| \right).
    \end{align*}
    On substituting this identity into the numerator of Eq.~\eqref{quotient_diff_one}, we obtain
    \begin{align}                                       &\left|\Lambda_{\widehat{\bm{v}}}(\bm{B}_1) - \Lambda_{\widehat{\bm{v}}}(\bm{B}_2) \right| \leq C \left( \left| \mathcal{S}_1 - \mathcal{S}_2 \right| + \left| \left| \widehat{\bm{v}} \circ \bm{\alpha}_{1}^{-1} \right|_{H^{1}(T_1)} - \left| \widehat{\bm{v}} \circ \bm{\alpha}_{2}^{-1} \right|_{H^{1}(T_2)}  \right| \right), \label{quotient_diff_two}
    \end{align}
    where
    \begin{align*}
        C = \frac{\left| \widehat{\bm{v}} \circ \bm{\alpha}_{1}^{-1} \right|_{H^{1}(T_1)} + \mathcal{S}_1 }{\mathcal{S}_1 \mathcal{S}_2}.
    \end{align*}
    Now, upon expanding the RHS of Eq.~\eqref{quotient_diff_two}, and extending the domains of integration for each seminorm over the domain $T_1 \cup T_2$, we obtain
    \begin{align}
        \nonumber &  \left|\Lambda_{\widehat{\bm{v}}}(\bm{B}_1) - \Lambda_{\widehat{\bm{v}}}(\bm{B}_2) \right| \\[1.5ex]
        \nonumber &\leq C \Bigg( \Bigg| \left\| \bm{\epsilon} \left(\widehat{\bm{v}} \circ \bm{\alpha}_{1}^{-1}\right) - \frac{1}{3} \text{div}(\widehat{\bm{v}} \circ \bm{\alpha}_{1}^{-1} ) \mathbb{I} \right\|_{L_{2}(T_1 \cup T_2) \times L_{2}(T_1 \cup T_2)} \\[1.5ex]
        &\nonumber - \left\| \bm{\epsilon} \left(\widehat{\bm{v}} \circ \bm{\alpha}_{2}^{-1}\right) - \frac{1}{3} \text{div}(\widehat{\bm{v}} \circ \bm{\alpha}_{2}^{-1} ) \mathbb{I} \right\|_{L_{2}(T_1 \cup T_2) \times L_{2}(T_1 \cup T_2)} \Bigg| \\[1.5ex]
        \nonumber &+\ell^{a} \left| \left| \int_{T_1 \cup T_2} \text{div}(\widehat{\bm{v}} \circ \bm{\alpha}_{1}^{-1}) \, dx \right| - \left| \int_{T_1 \cup T_2} \text{div}(\widehat{\bm{v}} \circ \bm{\alpha}_{2}^{-1}) \, dx \right| \right| \\[1.5ex]
        \nonumber &+ \ell^{a} \left| \left| \int_{T_1 \cup T_2} \text{curl}(\widehat{\bm{v}} \circ \bm{\alpha}_{1}^{-1} ) \, dx \right| - \left| \int_{T_1 \cup T_2} \text{curl}(\widehat{\bm{v}} \circ \bm{\alpha}_{2}^{-1} ) \, dx \right| \right| \\[1.5ex]
        \nonumber &+ \ell^{b} \left| \left| \int_{T_1 \cup T_2} \text{curl}(\text{curl} ( \widehat{\bm{v}} \circ \bm{\alpha}_{1}^{-1} )) \, dx \right| - \left| \int_{T_1 \cup T_2} \text{curl}(\text{curl} ( \widehat{\bm{v}} \circ \bm{\alpha}_{2}^{-1} )) \, dx \right| \right| \\[1.5ex]
        \nonumber &+ \left| \left| \widehat{\bm{v}} \circ \bm{\alpha}_{1}^{-1} \right|_{H^{1}(T_1 \cup T_2)} - \left| \widehat{\bm{v}} \circ \bm{\alpha}_{2}^{-1} \right|_{H^{1}(T_1 \cup T_2)}  \right| \Bigg) 
        \end{align}
        and furthermore
        \begin{align}
        \nonumber &  \left|\Lambda_{\widehat{\bm{v}}}(\bm{B}_1) - \Lambda_{\widehat{\bm{v}}}(\bm{B}_2) \right| \\[1.5ex]
        \nonumber &\leq C \Bigg(\left\| \bm{\epsilon} \left(\widehat{\bm{v}} \circ \bm{\alpha}_{1}^{-1} - \widehat{\bm{v}} \circ \bm{\alpha}_{2}^{-1}\right) - \frac{1}{3} \text{div}(\widehat{\bm{v}} \circ \bm{\alpha}_{1}^{-1} - \widehat{\bm{v}} \circ \bm{\alpha}_{2}^{-1} ) \mathbb{I} \right\|_{L_{2}(T_1 \cup T_2) \times L_{2}(T_1 \cup T_2)} \\[1.5ex]
        \nonumber &+ \ell^{a} \left( \left| \int_{T_1 \cup T_2} \text{div}(\widehat{\bm{v}} \circ \bm{\alpha}_{1}^{-1} - \widehat{\bm{v}} \circ \bm{\alpha}_{2}^{-1}) \, dx \right| + \left| \int_{T_1 \cup T_2} \text{curl}(\widehat{\bm{v}} \circ \bm{\alpha}_{1}^{-1} - \widehat{\bm{v}} \circ \bm{\alpha}_{2}^{-1}) \, dx \right| \right) \\[1.5ex]
        &+ \ell^{b} \left| \int_{T_1 \cup T_2} \text{curl}(\text{curl} ( \widehat{\bm{v}} \circ \bm{\alpha}_{1}^{-1} - \widehat{\bm{v}} \circ \bm{\alpha}_{2}^{-1} )) \, dx \right| + \left| \widehat{\bm{v}} \circ \bm{\alpha}_{1}^{-1} - \widehat{\bm{v}} \circ \bm{\alpha}_{2}^{-1} \right|_{H^{1}(T_1 \cup T_2)} \Bigg). \label{quotient_diff_three}
    \end{align}
    Here, we have used the reverse triangle inequality on the last three lines. Next, the first term on the RHS of Eq.~\eqref{quotient_diff_three} can be treated as follows
    \begin{align}
        \nonumber &\left\| \bm{\epsilon} \left(\widehat{\bm{v}} \circ \bm{\alpha}_{1}^{-1} - \widehat{\bm{v}} \circ \bm{\alpha}_{2}^{-1}\right) - \frac{1}{3} \text{div}(\widehat{\bm{v}} \circ \bm{\alpha}_{1}^{-1} - \widehat{\bm{v}} \circ \bm{\alpha}_{2}^{-1} ) \mathbb{I} \right\|_{L_{2}(T_1 \cup T_2) \times L_{2}(T_1 \cup T_2)} \\[1.5ex]
        &\leq C \left| \widehat{\bm{v}} \circ \bm{\alpha}_{1}^{-1} - \widehat{\bm{v}} \circ \bm{\alpha}_{2}^{-1} \right|_{H^{1}(T_1 \cup T_2)}, \label{semi_bound_zero} 
    \end{align}
    in accordance with Lemma~\ref{grad_lemma}.
    In addition, we have that
    \begin{align}
        \nonumber \left| \int_{T_1 \cup T_2} \text{div}(\widehat{\bm{v}} \circ \bm{\alpha}_{1}^{-1} - \widehat{\bm{v}} \circ \bm{\alpha}_{2}^{-1}) \, dx \right| &\leq \int_{T_1 \cup T_2} \left| \text{div}(\widehat{\bm{v}} \circ \bm{\alpha}_{1}^{-1} - \widehat{\bm{v}} \circ \bm{\alpha}_{2}^{-1}) \right| \, dx \\[1.5ex]
        &\leq C \sqrt{\left| T_1 \cup T_2 \right|} \left| \widehat{\bm{v}} \circ \bm{\alpha}_{1}^{-1} - \widehat{\bm{v}} \circ \bm{\alpha}_{2}^{-1} \right|_{H^{1}(T_1 \cup T_2)}, \label{semi_bound_one}
    \end{align}
    where we have used H{\"o}lder's inequality. Next, following similar arguments, we obtain
    \begin{align}
        \left| \int_{T_1 \cup T_2} \text{curl}(\widehat{\bm{v}} \circ \bm{\alpha}_{1}^{-1} - \widehat{\bm{v}} \circ \bm{\alpha}_{2}^{-1}) \, dx \right| &\leq C \sqrt{\left| T_1 \cup T_2 \right|} \left| \widehat{\bm{v}} \circ \bm{\alpha}_{1}^{-1} - \widehat{\bm{v}} \circ \bm{\alpha}_{2}^{-1} \right|_{H^{1}(T_1 \cup T_2)}, \label{semi_bound_two}
        \\[1.5ex]
        \left| \int_{T_1 \cup T_2} \text{curl}(\text{curl} ( \widehat{\bm{v}} \circ \bm{\alpha}_{1}^{-1} - \widehat{\bm{v}} \circ \bm{\alpha}_{2}^{-1} )) \, dx \right| &\leq C \sqrt{\left| T_1 \cup T_2 \right|} \vertiii{\widehat{\bm{v}} \circ \bm{\alpha}_{1}^{-1} - \widehat{\bm{v}} \circ \bm{\alpha}_{2}^{-1}}_{H^{2}(T_1 \cup T_2)}. \label{semi_bound_three}
    \end{align}
    Now, on substituting Eqs.~\eqref{semi_bound_zero}--\eqref{semi_bound_three} into Eq.~\eqref{quotient_diff_three}, we obtain
    \begin{align}
        \nonumber   \left|\Lambda_{\widehat{\bm{v}}}(\bm{B}_1) - \Lambda_{\widehat{\bm{v}}}(\bm{B}_2) \right|  &\leq C \Bigg(  \left(1 + \ell^{a} \sqrt{\left| T_1 \cup T_2 \right|} \right) \left| \widehat{\bm{v}} \circ \bm{\alpha}_{1}^{-1} - \widehat{\bm{v}} \circ \bm{\alpha}_{2}^{-1} \right|_{H^{1}(T_1 \cup T_2)}  \\[1.5ex]
        &+\ell^{b} \sqrt{\left| T_1 \cup T_2 \right|} \vertiii{\widehat{\bm{v}} \circ \bm{\alpha}_{1}^{-1} - \widehat{\bm{v}} \circ \bm{\alpha}_{2}^{-1}}_{H^{2}(T_1 \cup T_2)} \Bigg). \label{quotient_diff_four}
    \end{align}
    The expression above can be further simplified by using the following chain-rule identities
    \begin{align}
        \text{grad} (\widehat{\bm{v}} \circ \bm{\alpha}^{-1} ) &= \left( \widehat{\text{grad}}(\widehat{\bm{v}}) \circ \bm{\alpha}^{-1} \right) \bm{B}^{-1},  \\[1.5ex]
        \text{hess} (\widehat{v}_{i} \circ \bm{\alpha}^{-1} ) &= \bm{B}^{-T} \left( \widehat{\text{hess}}(\widehat{v}_{i}) \circ \bm{\alpha}^{-1} \right) \bm{B}^{-1},  \\[1.5ex]
        \int_{\widehat{T}} \widehat{\phi} \, d\widehat{x} &= \int_{T} \left( \widehat{\phi} \circ \bm{\alpha}^{-1} \right) \text{det}(\bm{B}^{-1}) \, dx,
    \end{align}
    where $\widehat{\phi}$ is a generic scalar function of $\widehat{\bm{x}}$ and
    \begin{align*}
        \left[\widehat{\text{grad}}(\widehat{\phi}) \right]_{i} = \frac{\partial \widehat{\phi}}{\partial \widehat{x}_{i}}, \qquad \left[\widehat{\text{hess}} (\widehat{\phi})\right]_{ij} = \frac{\partial^{2} \widehat{\phi}}{\partial \widehat{x}_i\partial \widehat{x}_j}.
    \end{align*}
    Based on these identities, we obtain the following upper bound for the H1 term in Eq.~\eqref{quotient_diff_four}
    \begin{align}
        \nonumber &\left| \widehat{\bm{v}} \circ \bm{\alpha}_{1}^{-1} - \widehat{\bm{v}} \circ \bm{\alpha}_{2}^{-1} \right|_{H^{1}(T_1 \cup T_2)} \\[1.5ex]
        \nonumber &= \left( \int_{T_1 \cup T_2} \left| \text{grad}( \widehat{\bm{v}} \circ \bm{\alpha}_{1}^{-1} - \widehat{\bm{v}} \circ \bm{\alpha}_{2}^{-1}) \right|^{2} dx \right)^{1/2} \\[1.5ex]
        \nonumber &= \left( \int_{T_1 \cup T_2} \left| \left( \widehat{\text{grad}}(\widehat{\bm{v}}) \circ \bm{\alpha}_{1}^{-1} \right) \bm{B}_{1}^{-1} - \left( \widehat{\text{grad}}(\widehat{\bm{v}}) \circ \bm{\alpha}_{2}^{-1} \right) \bm{B}_{2}^{-1} \right|^{2} dx \right)^{1/2} \\[1.5ex]
        \nonumber &\leq C \Bigg( \int_{T_1} \left| \left( \widehat{\text{grad}}(\widehat{\bm{v}}) \circ \bm{\alpha}_{1}^{-1} \right) \left(\bm{B}_{1}^{-1} - \bm{B}_{2}^{-1} \right) \right|^{2} dx  + \int_{T_2 \setminus T_1} \left| \left( \widehat{\text{grad}}(\widehat{\bm{v}}) \circ \bm{\alpha}_{2}^{-1} \right) \left( \bm{B}_{1}^{-1} - \bm{B}_{2}^{-1} \right) \right|^{2} dx \Bigg)^{1/2}
        \\[1.5ex]
        \nonumber & \leq C \Bigg( \frac{1}{\text{det}(\bm{B}_{1}^{-1})} \int_{\widehat{T}} \left|  \widehat{\text{grad}}(\widehat{\bm{v}}) \left(\bm{B}_{1}^{-1} - \bm{B}_{2}^{-1} \right) \right|^{2} d\widehat{x}  +\frac{1}{\text{det}(\bm{B}_{2}^{-1})}   \int_{\widehat{T}} \left| \widehat{\text{grad}}(\widehat{\bm{v}})  \left( \bm{B}_{1}^{-1} - \bm{B}_{2}^{-1} \right) \right|^{2} d\widehat{x} \Bigg)^{1/2} \\[1.5ex]
       \nonumber & \leq C \left(\frac{1}{\text{det}(\bm{B}_{1}^{-1})} + \frac{1}{\text{det}(\bm{B}_{2}^{-1})} \right)^{1/2} \left( \int_{\widehat{T}} \left|  \widehat{\text{grad}}(\widehat{\bm{v}}) \left(\bm{B}_{1}^{-1} - \bm{B}_{2}^{-1} \right) \right|^{2} d\widehat{x} \right)^{1/2} \\[1.5ex]
       & \nonumber \leq C \left(\frac{1}{\text{det}(\bm{B}_{1}^{-1})} + \frac{1}{\text{det}(\bm{B}_{2}^{-1})} \right)^{1/2} \left| \widehat{\bm{v}} \right|_{H^{1}(\widehat{T})} \left\| \bm{B}_{1}^{-1} - \bm{B}_{2}^{-1} \right\|_{\infty} \\[1.5ex]
       &\leq C \left(\frac{1}{\text{det}(\bm{B}_{1}^{-1})} + \frac{1}{\text{det}(\bm{B}_{2}^{-1})} \right)^{1/2} \left\| \bm{B}_{1}^{-1} - \bm{B}_{2}^{-1} \right\|_{F}. \label{H1_bound}
    \end{align}
    In the above equations, we have used the triangle inequality, root-mean-square-arithmetic-mean inequality, H{\"o}lder's inequality, and the equivalence of norms in a finite-dimensional space. 
    
    In a similar fashion, for the H2 term in Eq.~\eqref{quotient_diff_four} we obtain the following
    \begin{align}
        &\left| \widehat{\bm{v}} \circ \bm{\alpha}_{1}^{-1} - \widehat{\bm{v}} \circ \bm{\alpha}_{2}^{-1} \right|_{H^{2}(T_1 \cup T_2)} 
        \leq C \left(\frac{1}{\text{det}(\bm{B}_{1}^{-1})} + \frac{1}{\text{det}(\bm{B}_{2}^{-1})} \right)^{1/2} \left\| \bm{B}_{1}^{-1} - \bm{B}_{2}^{-1} \right\|_{F}^{2}. \label{H2_bound}
    \end{align}
    Upon substituting Eqs.~\eqref{H1_bound} and \eqref{H2_bound} into Eq.~\eqref{quotient_diff_four}, we obtain
    \begin{align}
        \nonumber &  \left|\Lambda_{\widehat{\bm{v}}}(\bm{B}_1) - \Lambda_{\widehat{\bm{v}}}(\bm{B}_2) \right| \\[1.5ex]
        \nonumber &\leq C \left(\frac{1}{\text{det}(\bm{B}_{1}^{-1})} + \frac{1}{\text{det}(\bm{B}_{2}^{-1})} \right)^{1/2} \\[1.5ex]
        &\times \Bigg(1 +   \ell^{a} \sqrt{\left| T_1 \cup T_2 \right|} +\ell^{b} \sqrt{\left| T_1 \cup T_2 \right|} \left\| \bm{B}_{1}^{-1} - \bm{B}_{2}^{-1} \right\|_{F} \Bigg) \left\| \bm{B}_{1}^{-1} - \bm{B}_{2}^{-1} \right\|_{F}.
        \label{quotient_diff_five}
    \end{align}
    This completes the proof.
\end{proof}

We are now ready to present the final result of this section.

\begin{theorem}
        Suppose that $\mathcal{T}$ is a valid mesh with non-degenerate elements. Then, for all $\bm{v} \in \left[H^{2} (T) \right]^3$ and $T \in \mathcal{T}$, the following Korn's inequality holds
        \begin{align}
        \left| \bm{v} \right|_{H^{1}(T)} \leq \kappa(\theta_{T}) \Bigg[ & \left\| \bm{\epsilon}(\bm{v}) - \frac{1}{3} \text{div} (\bm{v}) \mathbb{I} \right\|_{L_2(T) \times L_2(T)} + \ell^{a} \left| \int_{T} \text{div}(\bm{v}) \, dx \right| \label{local_korn} \\[1.5ex]
        \nonumber & + \ell^{a} \left| \int_{T} \text{curl}(\bm{v}) \, dx \right| + \ell^{b} \left| \int_{T} \text{curl}( \text{curl} ( \bm{v} )) \, dx \right| \Bigg],
        \end{align}
        where $\ell$ is a characteristic length scale associated with $\mathcal{T}$, $a$ and $b$ are constants that depend on the number of dimensions (3), and $\kappa: \mathbb{R}_{+} \rightarrow \mathbb{R}_{+}$ is a continuous, decreasing function of the minimum angle of each tetrahedron, $\theta_{T}$. \label{local_korn_lemma}
\end{theorem}

\begin{proof}
    We begin by introducing the following inequality
    \begin{align}
        \left| \bm{v} \right|_{H^{1} (T)} \leq k(T) \left( \left\| \bm{\epsilon} \left(\bm{v}\right) - \frac{1}{3} \text{div} \left(\bm{v} \right) \mathbb{I} \right\|_{L_2(T) \times L_2(T)} + \bm{\Psi}_{\ell} \left(\bm{v}\right) \right), \qquad \forall \bm{v} \in \left[ H^2(T) \right]^{3}, \label{initial_bound}
    \end{align}
    where $k(T)$ is the smallest possible number such that the inequality above holds. We note that this inequality is guaranteed to hold for a suitable family of constants, in accordance with Lemma~\ref{curl_curl_seminorm_lemma}.
    
    Following the arguments in~\cite{Brenner04}, we introduce a mapping between the physical domain $T$ and a reference domain $\widehat{T}$, such that $\bm{\alpha}: \widehat{T} \rightarrow T$, where 
    \begin{align*}
        \bm{\alpha} (\widehat{\bm{x}}) = \bm{B} \widehat{\bm{x}} + \bm{b}.
    \end{align*}
    Here, $\bm{B}$ is a matrix which belongs to the Lie group of non-singular matrices $GL(3)$, (as we discussed previously), and $\bm{b}$ is a constant vector. If we assume without loss of generality that $\bm{b} = \bm{0}$, then we have that $k(T) = k(\bm{B})$. It then follows that
    \begin{align*}
        k(\bm{B}) = \sup_{\substack{\widehat{\bm{v}} \in \left[H^2(\widehat{T})\right]^{3} \\
        \left\| \widehat{\bm{v}}\right\|_{H^{2}(\widehat{T})} = 1}} \left( \frac{\left| \widehat{\bm{v}} \circ \bm{\alpha}^{-1} \right|_{H^{1}(T)} }{\left\| \bm{\epsilon} \left(\widehat{\bm{v}} \circ \bm{\alpha}^{-1} \right) - \frac{1}{3} \text{div} \left(\widehat{\bm{v}} \circ \bm{\alpha}^{-1}  \right) \mathbb{I} \right\|_{L_2(T) \times L_2(T)} + \bm{\Psi}_{\ell} \left(\widehat{\bm{v}} \circ \bm{\alpha}^{-1} \right) } \right).
    \end{align*}
    We note that Lemma~\ref{continuity_lemma} in conjunction with the fact that $\left\| \widehat{\bm{v}}\right\|_{H^{2}(\widehat{T})} = 1$ imply that the RHS of the equation above represents a family of equicontinuous functions on $GL(3)$. It immediately follows that the supremum of this family $k(\cdot)$ is  continuous on $GL(3)$, (see~\cite{dieudonne2011foundations} for details). 
    
    The argument above holds for a generic element $T \in \mathcal{T}$. Of course, the mesh is usually composed of many such elements, and it remains for us to generalize our result to the entire mesh. Therefore, we consider the set $\left\{T_i : i \in I \right\}$, the family of tetrahedral elements which are affine homeomorphic to the reference element $\widehat{T}$, and we assume that all elements in the mesh belong to this set. Furthermore, we assume that the aspect ratios of all elements in the set are uniformly bounded, and (without loss of generality) that $\text{diam}(T_i) = 1$ for all $i$. Under these circumstances, the following inequalities hold in accordance with Theorem 3.1.3 of~\cite{ciarlet2002finite}
    \begin{align*}
        \left\| \bm{B}_{i} \right\| \leq C, \qquad \left\| \bm{B}^{-1}_{i} \right\| \leq C,
    \end{align*}
    where $\left\| \cdot \right\|$ is the matrix 2-norm for the element-specific transformation matrix $\bm{B}_i$. The inequalities above imply that the set of matrices $\left\{ \bm{B}_i: i \in I \right\}$ is a precompact subset of $GL(3)$. In turn, the continuity of $k(\cdot)$ on $GL(3)$ ensures the boundedness of $\left\{k(\bm{B}_i): i \in I \right\}$. As a result, the following inequality holds
    \begin{align*}
        \left| \bm{v} \right|_{H^{1} (T_i)} \leq C \left( \left\| \bm{\epsilon} \left(\bm{v}\right) - \frac{1}{3} \text{div} \left(\bm{v} \right) \mathbb{I} \right\|_{L_2(T_i) \times L_2(T_i)} + \bm{\Psi}_{\ell} \left(\bm{v}\right) \right), \qquad \forall \bm{v} \in \left[ H^2(T_i) \right]^{3},
    \end{align*}
    where $C$ is a constant independent of the particular tetrahedron $T_i$.
    
    Finally, the constant $k(\bm{B})$ can be expressed in terms of the minimum angle, $\theta_T$ of $T$. In a natural fashion, we can define the following set
    \begin{align*}
        S_{\theta} = \left\{k(T) : T \, \text{is a tetrahedron and the minimum angle of} \; T \, \text{is} \geq \theta > 0 \right\}.
    \end{align*}
    This set is bounded because $k(\bm{B})$ is bounded (as discussed above). We can define $\eta(\theta) = \sup S_{\theta}$, which is a non-negative, decreasing function of $\theta$, where $\eta: \mathbb{R}_{+} \rightarrow \mathbb{R}_{+}$. Now, in accordance with Eq.~\eqref{initial_bound}, we have that
    \begin{align}
        \left| \bm{v} \right|_{H^{1} (T)} \leq \eta(\theta_{T}) \left( \left\| \bm{\epsilon} \left(\bm{v}\right) - \frac{1}{3} \text{div} \left(\bm{v} \right) \mathbb{I} \right\|_{L_2(T) \times L_2(T)} + \bm{\Psi}_{\ell} \left(\bm{v}\right) \right), \qquad \forall \bm{v} \in \left[ H^2(T) \right]^{3}. \label{near_final_bound}
    \end{align}
    We complete the proof by defining a continuous, decreasing function $\kappa(\theta_{T}) \geq \eta(\theta_{T})$, and substituting this function into Eq.~\eqref{near_final_bound} in order to obtain the desired result (Eq.~\eqref{local_korn}). 
\end{proof}


\bibliographystyle{amsplain}
\bibliography{technical-refs}

\providecommand{\bysame}{\leavevmode\hbox to3em{\hrulefill}\thinspace}
\providecommand{\MR}{\relax\ifhmode\unskip\space\fi MR }
\providecommand{\MRhref}[2]{%
  \href{http://www.ams.org/mathscinet-getitem?mr=#1}{#2}
}
\providecommand{\href}[2]{#2}
\begin{thebibliography}{10}

\bibitem{adams2003sobolev}
Robert~A Adams and John~JF Fournier, \emph{Sobolev spaces}, Elsevier, 2003.

\bibitem{anderson1990modern}
John~David Anderson, \emph{Modern compressible flow: with historical
  perspective}, vol.~12, McGraw-Hill New York, 1990.

\bibitem{attia2005discrete}
Frank~S Attia and Gerhard Starke, \emph{On discrete {Korn}'s inequality for
  some nonconforming finite element approaches to linear elasticity}, PAMM:
  Proceedings in Applied Mathematics and Mechanics, vol.~5, Wiley Online
  Library, 2005, pp.~823--824.

\bibitem{breit2017trace}
Dominic Breit, Andrea Cianchi, and Lars Diening, \emph{Trace-free {Korn}
  inequalities in {Orlicz} spaces}, SIAM Journal on Mathematical Analysis
  \textbf{49} (2017), no.~4, 2496--2526.

\bibitem{brenner2015piecewise}
Susanne Brenner and Li-Yeng Sung, \emph{Piecewise {H1} functions and vector
  fields associated with meshes generated by independent refinements},
  Mathematics of Computation \textbf{84} (2015), no.~293, 1017--1036.

\bibitem{Brenner04}
Susanne~C Brenner, \emph{Korn's inequalities for piecewise {H1} vector fields},
  Mathematics of Computation (2004), 1067--1087.

\bibitem{brenner2008mathematical}
Susanne~C Brenner and L~Ridgway Scott, \emph{The mathematical theory of finite
  element methods}, vol.~3, Springer, 2008.

\bibitem{brenner2004poincare}
Susanne~C Brenner, Kening Wang, and Jie Zhao, \emph{Poincar{\'e}--{Friedrichs}
  inequalities for piecewise {H2} functions}, Numerical Functional Analysis and
  Optimization \textbf{25} (2004), no.~5-6, 463--478.

\bibitem{chen2020versatile}
Xi~Chen and David~M Williams, \emph{Versatile mixed methods for the
  incompressible {Navier}--{Stokes} equations}, Computers \& Mathematics with
  Applications \textbf{80} (2020), no.~6, 1555--1577.

\bibitem{ciarlet2002finite}
Philippe~G Ciarlet, \emph{The finite element method for elliptic problems},
  SIAM, 2002.

\bibitem{dain2006generalized}
Sergio Dain, \emph{Generalized {Korn}'s inequality and conformal {Killing}
  vectors}, Calculus of Variations and Partial Differential Equations
  \textbf{25} (2006), no.~4, 535--540.

\bibitem{di2015hybrid}
Daniele~A Di~Pietro and Alexandre Ern, \emph{A hybrid high-order locking-free
  method for linear elasticity on general meshes}, Computer Methods in Applied
  Mechanics and Engineering \textbf{283} (2015), 1--21.

\bibitem{di2013locking}
Daniele~A Di~Pietro and Serge Nicaise, \emph{A locking-free discontinuous
  {Galerkin} method for linear elasticity in locally nearly incompressible
  heterogeneous media}, Applied Numerical Mathematics \textbf{63} (2013),
  105--116.

\bibitem{di2011mathematical}
Daniele~Antonio Di~Pietro and Alexandre Ern, \emph{Mathematical aspects of
  discontinuous {Galerkin} methods}, vol.~69, Springer Science \& Business
  Media, 2011.

\bibitem{dieudonne2011foundations}
Jean Dieudonn{\'e}, \emph{Foundations of modern analysis}, Read Books Ltd,
  2011.

\bibitem{evans2013isogeometricA}
John~A Evans and Thomas~JR Hughes, \emph{Isogeometric divergence-conforming
  {B}-splines for the steady {Navier}--{Stokes} equations}, Mathematical Models
  and Methods in Applied Sciences \textbf{23} (2013), no.~08, 1421--1478.

\bibitem{evans2013isogeometricB}
\bysame, \emph{Isogeometric divergence-conforming {B}-splines for the unsteady
  {Navier}--{Stokes} equations}, Journal of Computational Physics \textbf{241}
  (2013), 141--167.

\bibitem{faraco2005geometric}
Daniel Faraco and Xiao Zhong, \emph{Geometric rigidity of conformal matrices},
  Annali della Scuola Normale Superiore di Pisa-Classe di Scienze \textbf{4}
  (2005), no.~4, 557--585.

\bibitem{feireisl2009singular}
Eduard Feireisl and Anton{\'\i}n Novotn{\`y}, \emph{Singular limits in
  thermodynamics of viscous fluids}, vol.~2, Springer, 2009.

\bibitem{fu2018strongly}
Guosheng Fu and Christoph Lehrenfeld, \emph{A strongly conservative hybrid
  {DG}/mixed {FEM} for the coupling of {Stokes} and {Darcy} flow}, Journal of
  Scientific Computing \textbf{77} (2018), no.~3, 1605--1620.

\bibitem{fuchs2009application}
Martin Fuchs and Oliver Schirra, \emph{An application of a new coercive
  inequality to variational problems studied in general relativity and in
  {Cosserat} elasticity giving the smoothness of minimizers}, Archiv der
  Mathematik \textbf{93} (2009), no.~6, 587.

\bibitem{hong2016presentation}
Qingguo Hong, \emph{A discrete {Korn}’s inequality and related finite
  elements}, Presented at AANMPDE Duisburg-Essen University, Essen, Germany,
  2016.

\bibitem{hong2016uniformly}
Qingguo Hong and Johannes Kraus, \emph{Uniformly stable discontinuous
  {Galerkin} discretization and robust iterative solution methods for the
  {Brinkman} problem}, SIAM Journal on Numerical Analysis \textbf{54} (2016),
  no.~5, 2750--2774.

\bibitem{hong2016robust}
Qingguo Hong, Johannes Kraus, Jinchao Xu, and Ludmil Zikatanov, \emph{A robust
  multigrid method for discontinuous {Galerkin} discretizations of {Stokes} and
  linear elasticity equations}, Numerische Mathematik \textbf{132} (2016),
  no.~1, 23--49.

\bibitem{horgan1995korn}
Cornelius~O Horgan, \emph{Korn’s inequalities and their applications in
  continuum mechanics}, SIAM review \textbf{37} (1995), no.~4, 491--511.

\bibitem{john2016finite}
Volker John, \emph{Finite element methods for incompressible flow problems},
  Springer, 2016.

\bibitem{knobloch2003korn}
Petr Knobloch and Lutz Tobiska, \emph{On {Korn}'s first inequality for
  quadrilateral nonconformimg finite elements of first order approximation
  properties}, Univ., Fak. f{\"u}r Mathematik, 2003.

\bibitem{kreyszig1991introductory}
Erwin Kreyszig, \emph{Introductory functional analysis with applications},
  vol.~17, John Wiley \& Sons, 1991.

\bibitem{lee2018robust}
Jeonghun~J Lee, \emph{Robust three-field finite element methods for {Biot}’s
  consolidation model in poroelasticity}, BIT Numerical Mathematics \textbf{58}
  (2018), no.~2, 347--372.

\bibitem{mardal2006observation}
Kent-Andre Mardal and Ragnar Winther, \emph{An observation on {Korn}’s
  inequality for nonconforming finite element methods}, Mathematics of
  Computation \textbf{75} (2006), no.~253, 1--6.

\bibitem{miller2020versatile}
Edward~A Miller, Xi~Chen, and David~M Williams, \emph{Versatile mixed methods
  for non-isothermal incompressible flows}, arXiv preprint arXiv:2007.08679
  (2020).

\bibitem{riviere2008discontinuous}
B{\'e}atrice Rivi{\`e}re, \emph{Discontinuous {Galerkin} methods for solving
  elliptic and parabolic equations: theory and implementation}, SIAM, 2008.

\bibitem{schirra2009regularity}
O~Schirra, \emph{Regularity results for solutions to variational problems with
  applications in fluid mechanics and general relativity}, Ph.D. thesis,
  Doctoral thesis, Saarland University, Saarbr{\"u}cken, 2009.

\bibitem{schirra2012new}
Oliver~D Schirra, \emph{New {Korn}-type inequalities and regularity of
  solutions to linear elliptic systems and anisotropic variational problems
  involving the trace-free part of the symmetric gradient}, Calculus of
  Variations and Partial Differential Equations \textbf{43} (2012), no.~1-2,
  147--172.

\bibitem{zhang2017discrete}
Sheng Zhang, \emph{Discrete {Korn}’s inequality for shells}, Calcolo
  \textbf{54} (2017), no.~1, 243--265.

\end{thebibliography}

\end{document}